\documentclass[12pt,a4paper]{amsart}

\usepackage[utf8]{inputenc}
\usepackage[hmargin=30mm, vmargin=25mm, includefoot, twoside]{geometry}
\usepackage[usenames,dvipsnames]{xcolor}
\definecolor{darkred}{RGB}{150,0,0}
\definecolor{darkblue}{RGB}{0,0,150}
\usepackage[pdfusetitle]{hyperref}
\hypersetup{
	bookmarksopen=true,
	colorlinks = true,
	urlcolor   = darkblue,
	linkcolor  = darkblue,
	citecolor  = darkred,
	pdfauthor={Goulnara Arzhantseva, Dawid Kielak, Tim de Laat, Damian Sawicki}
}

\usepackage[alphabetic]{amsrefs}
\usepackage{amsfonts,amssymb,verbatim,centernot}
\usepackage{latexsym}
\usepackage{mathrsfs}
\usepackage{stmaryrd}
\usepackage{xspace}
\usepackage{enumerate, paralist}
\usepackage{graphicx}
\usepackage[all,cmtip]{xy}
\usepackage[all]{xypic}
\usepackage[capitalize]{cleveref}
\usepackage[normalem]{ulem}
\usepackage{txfonts, pxfonts}

\usepackage{mathtools}

\definecolor{myblue1}{RGB}{35,119,189}

\newtheorem{thmA}{Theorem}

\crefname{thmA}{Theorem}{Theorems}
\newtheorem{corA}[thmA]{Corollary}
\newtheorem{conjA}[thmA]{Conjecture}

\newtheorem{thm}{Theorem}[section]
\crefname{thm}{Theorem}{Theorems}
\newtheorem{cor}[thm]{Corollary}
\newtheorem{lem}[thm]{Lemma}

\newtheorem{prop}[thm]{Proposition}

\theoremstyle{definition}
\newtheorem{defn}[thm]{Definition}
\crefname{defn}{Definition}{Definitions}
\crefname{figure}{Figure}{Figures}
\newtheorem{conv}[thm]{Convention}

\newtheorem*{cla*}{Claim}

\newtheorem*{Ack}{Acknowledgements}

\theoremstyle{remark}
\newtheorem{rem}[thm]{Remark}
\newtheorem{ex}[thm]{Example}

\numberwithin{equation}{section}

\newcommand{\T}{\mathbb{T}}
\newcommand{\Z}{\mathbb{Z}}

\newcommand{\F}{\mathbb{F}}

\newcommand{\Sp}{\mathbb{S}}
\newcommand{\N}{\mathbb{N}}
\newcommand{\R}{\mathbb{R}}
\newcommand{\C}{\mathbb{C}}
\newcommand{\QQ}{\mathbb{Q}}

\newcommand{\HHH}{\mathbb{H}}

\newcommand{\eps}{\varepsilon}

\newcommand{\supp}{\textnormal{Supp}}

\newcommand{\id}{\textnormal{id}}

\DeclareMathOperator{\Aff}{Aff}
\DeclareMathOperator{\cay}{Cay}

\DeclareMathOperator{\GL}{GL}
\DeclareMathOperator{\SL}{SL}
\DeclareMathOperator{\PSL}{PSL}
\DeclareMathOperator{\SO}{SO}
\DeclareMathOperator{\diam}{diam}
\newcommand*{\defeq}{\mathrel{\vcenter{\baselineskip0.5ex \lineskiplimit0pt
                     \hbox{\scriptsize.}\hbox{\scriptsize.}}}%
                     =}

\def\cO{{\mathcal O}}
\def\acts{\curvearrowright}

\newcounter{dawidcomments}

\newcounter{damiancomments}

\newcounter{goulnaracomments}

\definecolor{darkgreen}{rgb}{0.0, 0.5, 0.0}
\newcounter{timcomments}

\begin{document}
	\title{Spectral gap and origami expanders}

\author{Goulnara Arzhantseva}
\address{\hspace{-3ex} Universit\"at Wien, Fakult\"at f\"ur Mathematik\\
	Oskar-Morgenstern-Platz 1, 1090 Wien, Austria}
\email{goulnara.arzhantseva@univie.ac.at}

\author{Dawid Kielak}
\address{\hspace{-3ex} Mathematical Institute, University of Oxford, Andrew Wiles Building, Radcliffe Observatory Quarter, Woodstock Road, Oxford
	OX2 6GG, United Kingdom}
\email{kielak@maths.ox.ac.uk}

\author{Tim de Laat}
\address{\hspace{-3ex} University of M\"unster, Mathematical Institute, Einsteinstra\ss{}e 62, 48149 M\"unster, Germany}
\email{tim.delaat@uni-muenster.de}

\author{Damian Sawicki}
\address{\hspace{-3ex} KU Leuven, Department of Mathematics,
	Celestijnenlaan 200b – box 2400, 3001 Leuven, Belgium}
\urladdr{https://sites.google.com/view/damiansawicki}

\date{}

\thanks{G.A.\ was partially supported by the ERC grant ANALYTIC no.\ 259527. D.K.\ was supported by the ERC grant FIBRING no.\ 850930. T.dL.\ was supported by the DFG -- Project-ID 427320536 -- SFB 1442, and under Germany's Excellence Strategy EXC 2044 390685587, Mathematics Münster: Dynamics--Geometry--Structure. D.S.\ was supported by the FWO research project G090420N of the Research Foundation Flanders.}

\baselineskip=16pt

\begin{abstract}
	We construct the first measure-preserving affine actions with spectral gap on surfaces of arbitrary genus $g > 1$. We achieve this by finding geometric representatives of  multi-twists on origami surfaces. As a major application, we construct
	new expanders that are coarsely distinct from the classical expanders obtained via the Laplacian as Cayley graphs of finite quotients of a group. Our methods also show
	that the Margulis expander, and hence the Gabber--Galil expander, is coarsely distinct from the Selberg expander.
\end{abstract}

\maketitle

\section{Introduction} \label{sec:introduction}

Spectral gap is a fundamental property  at the heart of many rigidity results in various areas of modern mathematics and computer science.
In the setting of measure-preserving group actions, it plays a significant role  in the solution of the Banach--Ruziewicz problem \cites{Margulis,sullivan,drinfeld}, in the study of random transformations \cites{furman-shalom,conze-guivarch}, and in orbit equivalence rigidity results \cites{ioana1,ioana2}. Moreover, it underlies various constructions of expanders, which will be explained in detail below.

Let $G$ be a finitely generated group with a finite symmetric generating set~$S$ containing the identity element $1_G$ of $G$. Let $(X,\mu)$ be a probability space, $\rho \colon G \acts (X,\mu)$ a measure-preserving action, and $\pi_{\rho} \colon G \to \mathcal{U}(L^2(X,\mu))$ the associated Koopman representation, which is the unitary representation given by $(\pi_{\rho}(g)f)(x) = f(g^{-1}x)$. The space of constant functions on $X$ is an invariant subspace of $\pi_{\rho}$. Its orthogonal complement is the subspace $L^2_0(X,\mu)$ of $L^2(X,\mu)$ consisting of the elements of $L^2(X,\mu)$ with mean zero. For convenience, we will denote $\pi_{\rho}(g)f$ by $g.f$.

\begin{defn}[Spectral gap] \label{defn:spectralgap}
A probability-measure-preserving action $\rho \colon G \acts (X,\mu)$ has \emph{spectral gap} if
there exists a constant $\kappa > 0$ such that for every $f\in L^2_0(X,\mu)$,
\[
 	\sum_{s\in S}\| s.f - f \|_2\geqslant \kappa \| f \|_2.
\]
\end{defn}

In this article, we construct the very first continuous actions with spectral gap of free groups on surfaces of arbitrary genus $g > 1$. Concretely, the surfaces we consider are origami surfaces, also called square-tiled translation surfaces. These surfaces are obtained by gluing together finitely many copies of the unit square along their edges in such a way that the resulting surface is a branched covering of the torus $\mathbb{T}^2$; see \cref{defn:origami}.
Given an arbitrary origami surface $\Sigma$ of genus $g \geqslant 1 $, we explicitly construct a measure-preserving action $\mathbb{F}_2 \curvearrowright \Sigma$ that is \emph{expanding in measure}. The latter notion is a convenient characterisation of spectral gap for measure-preserving actions; see \cref{defn:measexp} and \cref{prop:meSG}. This concept was studied systematically in~\cite{FV}.

\begin{thmA}\label{thm:F2}
	Let $\Sigma$ be an arbitrary origami surface. Then the free group $\F_2$ admits a measure-preserving action by Lipschitz homeomorphisms on $\Sigma$ that is expanding in measure. Equivalently, this action $\F_2 \curvearrowright \Sigma$ has spectral gap.
\end{thmA}

In fact, we prove more. For every $\Sigma$ as in \cref{thm:F2} and every $m\in \N_{>0}$, one can restrict the action $\F_2 \acts \Sigma$ to a certain infinite-index subgroup $H_m < \F_2$ while retaining the spectral gap. In detail, this can be achieved by replacing the matrices $a_k, b_k$ in the construction of the action in \cref{subsec:actionsonsurfaces} with $a_{mk},b_{mk}$. Furthermore, if~$m$ divides $m'$ with $m\neq m'$, then $H_{m'}$ is an infinite-index subgroup of $H_m$. Each $H_m$ is itself isomorphic to $\F_2$, so for every fixed origami surface $\Sigma$, we obtain infinitely many
actions with spectral gap of $\F_2$ on $\Sigma$.

\Cref{thm:F2} gives the first examples of continuous measure-preserving actions with spectral gap on surfaces of higher genus $g$, i.e.,~genus $g > 1$. The only prior examples of actions with spectral gap on surfaces were on the sphere $\Sp^2$~\cites{drinfeld,gjs,BGSU2} and on the torus $\T^2$~\cites{rosenblatt,GG,LR}.

In order to prove \cref{thm:F2}, we first establish a general result that allows us to lift measure expansion to extensions of dynamical systems; see \cref{schmidt trick}.

\smallskip

Given an arbitrary translation surface $\Sigma$, let $\Aff^+(\Sigma)$ denote the group of all affine orientation-preserving homeomorphisms (see \cite{Edwardsetal2022} for the definitions and a modern treatment of these topics). The derivative of any $u\in \Aff^+(\Sigma)$ is constant, so there is a well-defined homomorphism $\Aff^+(\Sigma) \to \SL_2(\R)$,
whose image $\SL(\Sigma)$ is called the Veech group of $\Sigma$. The Veech group is a discrete and non-cocompact subgroup of $\SL_2(\R)$, and if the genus of $\Sigma$ is larger than $1$, the kernel $N(\Sigma)$ of $\Aff^+(\Sigma) \to \SL_2(\R)$ is finite \cite{Veech}*{Section~2}. If the Veech group is finitely generated, as a non-cocompact Fuchsian group, it must be virtually free (this class includes virtually $\mathbb Z$ and finite groups).

If $\Sigma$ is, moreover, an origami surface, the Veech group is a lattice \cite{GutkinJudge}*{Theorem~5.5}. Hence, $\Aff^+(\Sigma)$ is finitely generated whenever $\Sigma$ is an origami surface of genus larger than $1$. (Clearly, $\Aff^+(\T^2)$ contains~$\T^2$ as a subgroup, so finite generation is not true in the genus~$1$ case.)

\cref{thm:F2} asserts that a certain subgroup of $\Aff^+(\Sigma)$ acts with spectral gap. Hence, the defining action of $\Aff^+(\Sigma)$ on any origami surface $\Sigma$ of genus $> 1$ has spectral gap. Indeed, if a finitely generated subgroup of a finitely generated group acts with spectral gap, then the action of the overgroup has spectral gap as well.

\begin{corA}
For every origami surface~$\Sigma$ of genus $> 1$, the defining action of $\Aff^+(\Sigma)$ on $\Sigma$ has spectral gap.
\end{corA}

Our results suggest many directions for further study. One such line of research is about how special origami surfaces and the actions on them that we construct in the proof of \cref{thm:F2} are. In fact, we expect spectral gap of affine actions on arbitrary translation surfaces.

\begin{conjA}
	Let $\Sigma$ be a translation surface such that $\Aff^+(\Sigma)$ is finitely generated and not virtually cyclic. Then
	the action of $\Aff^+(\Sigma)$ on $\Sigma$ has spectral gap.
\end{conjA}

We exclude the virtually cyclic case, since such groups cannot act with spectral gap on a non-atomic measure space. It is in fact an open problem whether there is a translation surface whose Veech group is cyclic and generated by a hyperbolic element~\cite{HMSZ}*{Problem 6}.

\smallskip

As mentioned before, actions with spectral gap can be used to construct expanders. Expanders are sequences of finite, highly connected, sparse graphs with an increasing number of vertices. They play a prominent role in various areas of mathematics and computer science. We refer to \cites{HLW, LubotzkyBook, LubotzkySurvey, TaoBook, KowalskiBook} for thorough surveys on expanders and their wide range of applications.

Let $\Gamma=(V,E)$ be a finite, connected, simple (i.e.,~without multiple edges or loops) graph. Given a subset $A \subseteq V$, let $\partial A$ denote the edge boundary of $A$, i.e.,~the subset of $E$ consisting of edges joining a vertex in $A$ and a vertex in $V \smallsetminus A$. The Cheeger constant $h(\Gamma)$, which is a measure of the connectivity of the graph $\Gamma$, is defined by
\[
h(\Gamma)\coloneqq \min \left\{ \frac{|\partial A|}{|A|} \;\bigg\vert\; A \subseteq V,\; 1 \leqslant |A| \leqslant \frac{1}{2}|V| \right\}.
\]
The degree $\mathrm{deg}(v)$ of a vertex $v \in V$ is the number of edges $e \in E$ that contain $v$.
\begin{defn}[Expander] \label{defn:expander}
	An \emph{expander} is a sequence $\left(\Gamma_n\right)_{n}$ of finite  graphs $\Gamma_n=(V_n,E_n)$ with $\lim_{n \to \infty} |V_n| = \infty$ such that there exist $\varepsilon > 0$ and $D \geqslant 1$ such that for all $n \in \mathbb{N}$, we have $h(\Gamma_n) \geqslant \varepsilon$ and $\max_{v \in V_n} \mathrm{deg}(v) \leqslant D$.
\end{defn}

Initially, the existence of expanders was shown by probabilistic methods~\cites{kolmogorovbarzdin, pinsker}, but over the years, various explicit constructions have been described.

The first explicit construction of expanders is due to Margulis \cite{Mar73}. His expander is a sequence $(M_n)_{n}$ of bipartite graphs, where the vertex set of $M_n$ is given by $(\mathbb{Z}_n \times \mathbb{Z}_n) \sqcup (\mathbb{Z}_n \times \mathbb{Z}_n)$ (with the obvious bipartition and with $\Z_n\coloneqq \Z/n\Z$) and in which the degree of every vertex is $5$. Margulis's proof of expansion relies on methods from representation theory.
Based on this, other, in general non-bipartite, examples of expanders were obtained from finite quotients of groups with Kazhdan's Property~(T)  \cite{AlonMilman1985}.  A modification of Margulis's construction, which gives a sequence
quasi-isometric to Margulis's expander, was described in \cite{GG}, where expansion was established by means of Fourier theory on the continuous torus $\mathbb{T}^2=\mathbb{R}^2 / \mathbb{Z}^2$. In turn, this construction was analysed further in \cite{JM}, which relies on Fourier theory on the
discrete torus $\mathbb{Z}_n \times \mathbb{Z}_n$.

In \cite{FV}, Vigolo described systematically how to obtain expanders from actions with spectral gap (under mild regularity assumptions) by means of graphs approximating the action.
These approximating graphs are quasi-isometric to the level sets of a geometric object: the warped cone over the action; see \cref{defn:warpedcone}. The idea of measure expansion was already considered before Vigolo's work, e.g.\ in \cites{BY,GMP,BIG}, and as mentioned above, already the original Margulis expander and several other related classical expanders come from suitable measure-preserving transformations with spectral gap.
Since every sufficiently regular action with spectral gap gives rise to an expander, \cref{thm:F2} provides a wide family of expanders.

\begin{corA}\label{cor:origamiE}
	For every origami surface~$\Sigma$ and every action $\F_2\acts \Sigma$ as in \cref{thm:F2}, there is an associated expander ${(\Gamma_n)}_n$
	quasi-isometric to the warped cone~$\cO_{\F_2}\Sigma$ over the action $\F_2\acts \Sigma$.
\end{corA}

\smallskip

In the second part of the article, we analyse these expanders, which are examples of
\emph{origami expanders} as introduced in \cref{def:origamiexpander}. In particular, we study them from the point of view of coarse geometry. We investigate their coarse equivalence classes and we prove that our construction indeed yields a great diversity of new expanders.

\begin{thmA}\label{differentFromSelberg}
	Let $\Sigma$ be an arbitrary origami surface and ${(\Gamma_n)}_n$ an expander as in \cref{cor:origamiE}. Then ${(\Gamma_n)}_n$ does not admit a
	coarse embedding into any Selberg-type expander.
\end{thmA}

Roughly speaking, by a Selberg-type expander we mean an expander that is a sequence of graphs obtained as finite quotients of a  finitely generated discrete subgroup of $\PSL_2(\mathbb R)$ or of $\SL_2(\mathbb R)$; see \cref{def:selberg}. Therefore, \cref{differentFromSelberg} asserts that while both our origami expanders and Selberg-type expanders are intimately related to the hyperbolic plane, these two large families of expanders are completely different.

In fact, after relaxing the conclusion of \cref{differentFromSelberg} from the lack of a coarse embedding to the lack of a coarse equivalence (or equivalently, to the lack of a quasi-isometry), we show that many of our origami expanders are even distinct from \emph{any} box space; see \cref{def:boxspace}.

\begin{thmA}\label{fundamentalVanishing} Let $g > 1$, and let $Z_g$ be the staircase of genus $g$ as in \cref{staircaseDef} equipped with the action of $\F_2$ from \cref{thm:F2}. Let ${(\Gamma_n)}_n$ be an expander associated with this action as in \cref{cor:origamiE}.
	The discrete fundamental groups $\pi_{1,r}(\Gamma_n)$ at scale $r$ vanish for $r$ and $n$ sufficiently large. Consequently:
	\begin{enumerate}[(i)]
		\item the expander ${(\Gamma_n)}_n$ is not
		coarsely equivalent to any box space;
		\item the expander ${(\Gamma_n)}_n$ is not
		coarsely equivalent to any warped cone $\cO_H Y$, where $H$ is an infinite group, $Y$ is a manifold, and the action $H \acts Y$ is free.
	\end{enumerate}
\end{thmA}

Once we have proved that our origami expanders are distinct from the previously known expanders mentioned above, a natural problem is to distinguish origami expanders from each other, for different surfaces and actions. This problem is highly non-trivial. The strongest to date coarse rigidity results obtained for box spaces \cite{DK} and warped cones \cites{VigFund,FNvL} make use of discrete fundamental groups, which vanish for many origami expanders by \cref{fundamentalVanishing}. Other standard coarse invariants, such as Property A or asymptotic dimension, are useless for expanders. Moreover, since all our origami expanders (as in \cref{cor:origamiE}) come from warped cones over actions of a fixed group on manifolds of a fixed dimension, and the actions are essentially free, we expect their ``local'' properties around free orbits to be the same. \Cref{hohoho} indeed provides an example of such a local property shared among our origami expanders.

Instead of distinguishing our origami expanders per se, we distinguish them together with the corresponding bundle. More precisely, our origami expander~${(\Gamma_n)}_n$ is
quasi-isometric to the warped cone $\cO_{\F_2}\Sigma$, and together with ${(\Gamma_n)}_n$, we consider the warped cone $\cO_{\SL_2(\Z)}\T^2$, which is
quasi-isometric to the original
Margulis expander. Induced by the canonical branched covering $\rho\colon \Sigma\to \T^2$, there is a sequence \label{sequenceAsFunctionUsage} $P^\rho\colon\cO_{\F_2} \Sigma \to \cO_{\SL_2(\Z)}\T^2$ of coarse maps. We will study the sequence ${(\Gamma_n)}_n$ of graphs together with this sequence of maps. We call two such sequences $P^\rho\colon\cO_{\F_2} \Sigma \to \cO_{\SL_2(\Z)}\T^2$ and $P^{\rho'}\colon\cO_{\F_2} \Sigma' \to \cO_{\SL_2(\Z)}\T^2$
coarsely equivalent if there are
coarse equivalences $\cO_{\F_2} \Sigma\simeq \cO_{\F_2} \Sigma'$ and $\cO_{\SL_2(\Z)}\T^2\simeq \cO_{\SL_2(\Z)}\T^2$ such that the corresponding diagram commutes; see  \cref{bundles} for details.

\begin{thmA}\label{pairwiseDistinctExpanderBundles} For every origami surface $\Sigma$, the sequence $P^\rho\colon\cO_{\F_2} \Sigma \to \cO_{\SL_2(\Z)}\T^2$ of maps from the origami expander from \cref{cor:origamiE} to the original Margulis expander is not a coarse equivalence.
	
	Furthermore, there is an explicit infinite set $I\subseteq \N$ of genera such that for the corresponding staircases $Z_g$ with $g\in I$, the sequences $P^\rho\colon\cO_{\F_2} Z_g \to \cO_{\SL_2(\Z)}\T^2$ are pairwise not coarsely equivalent.
\end{thmA}

\begin{Ack}
	The authors are grateful to Marc Burger, \L{}ukasz Grabowski, and Tadeusz Januszkiewicz for valuable discussions and suggestions at an early stage of the project. They also thank Benson Farb, David Hume, and Alessandro Sisto  for pointing out \cite{EskinFarb1997}, \cite{Li}, and \cite{KL}, respectively. The project got momentum when G.A., D.K.~and T.dL.~participated in the workshop \emph{Geometric Structures in Group Theory} at the Mathematisches Forschungsinstitut Oberwolfach (MFO) in 2020. The three authors thank the organisers for the invitation and the MFO for the excellent working conditions.
	The authors would also like to thank the referee for multiple valuable comments.
\end{Ack}

\section{Preliminaries}

\subsection{Ergodicity and strong ergodicity}
Let $G$ be a finitely generated group with a finite symmetric generating set $S$ containing the identity element $1_G$ of $G$, which is often simply denoted by $1$.

\begin{defn}[Measure-preserving action]
An action $G\acts (X,\mu)\colon (g,x)\mapsto g.x$ of $G$ on a measure space $(X,\mu)$ is called \emph{measurable} if for every $g \in G$, the map $X \to X,\;x \mapsto g.x$ is measurable. A measurable action $G\acts (X,\mu)$ is called \emph{measure preserving} if for every measurable $U\subseteq X$ and every $g \in G$, we have $\mu(U) = \mu(g.U)$.
\end{defn}

\begin{defn}[Invariant / almost invariant subsets]\label{defi:invariant}
Let $G\acts (X,\mu)$ be a measure-preserving action. A measurable subset $U \subseteq X$ is \emph{invariant} if $g.U=U$ for every $g \in G$. (This implies $\mu(g.U \smallsetminus U)=0$.)  A sequence ${(U_n)}_n$ of measurable subsets is \emph{almost invariant} if for every $g \in G$,
\[
	\lim_n \mu( g.U_n \smallsetminus U_n) = 0.
\]
\end{defn}

Observe that
if $g = s_1\dots s_k$ with $s_i\in S$, then
\begin{gather}\label{generators-chaining}
\begin{aligned}
g.U \smallsetminus U &= s_1\dots s_k. U \smallsetminus U \\
&\subseteq
s_1\dots s_{k-1}.(s_k. U \smallsetminus U) \cup s_1\dots s_{k-2}.(s_{k-1}. U \smallsetminus U) \cup \cdots \cup (s_1. U \smallsetminus U).
\end{aligned}
\end{gather}
Therefore, it suffices to check invariance and almost invariance of subsets under the action by generators only, and these properties do not depend on the choice of the finite generating set.

\begin{defn}[Ergodic / strongly ergodic action]
A measure-preserving action $G \curvearrowright (X,\mu)$ on a measure space with $\mu(X)<\infty$ is \emph{ergodic} if every invariant subset $U\subseteq X$ satisfies
$\mu(U)\big( \mu(X) - \mu(U) \big) = 0$. A measure-preserving action $G \curvearrowright X$ on a measure space with $\mu(X)<\infty$ is \emph{strongly ergodic} if every almost invariant sequence ${(U_n)}_n$ satisfies $$\lim_n \mu(U_n)\big( \mu(X) - \mu(U_n) \big) = 0.$$
\end{defn}

\subsection{Spectral gap and measure expansion}
Spectral gap was defined in \cref{defn:spectralgap}. The definition of spectral gap is independent of the choice of the finite generating set~$S$.

We now recall a characterisation of spectral gap for measure-preserving actions. This equivalent notion was defined explicitly in \cite{FV} under the name ``measure expansion''. It can be defined for measurable actions that are not necessarily measure preserving, but in this article, we focus on the measure-preserving case. Similar phenomena had been studied before in~\cites{BY,GMP,BIG}.

\begin{defn}[Measure-expanding action] \label{defn:measexp}
A measure-preserving action $G \acts (X,\mu)$ is \emph{measure expanding} (or \emph{expanding in measure}) if there exists a constant $c > 0$ such that for every measurable subset $U \subseteq X$ of finite measure satisfying $0 < \mu(U) \leqslant \frac{1}{2}\mu(X)$, we have
\[
	\mu(SU)  \geqslant (1 + c)\mu(U),
\]
 where $SU\coloneqq \bigcup_{s \in S} s.U.$
\end{defn}
The inequality $\mu(SU)  \geqslant (1 + c)\mu(U)$ is equivalent to $\mu(SU \smallsetminus U)  \geqslant c\mu(U)$, which is a strong form of non-invariance of the set $U$.

Using \eqref{generators-chaining}, one checks that measure expansion is independent of the choice of the finite generating set, though the expansion constant $c$ might depend on it.

\begin{prop} \label{prop:meSG}
A probability-measure-preserving action $\rho\colon G \acts (X,\mu)$ is measure expanding if and only if it has spectral gap.
\end{prop}
The proof essentially follows from~\cite{schmidtSE}, which uses different terminology. Indeed, the ``if'' implication follows by a standard argument from \cref{defn:spectralgap} applied to the function $\chi_U - \mu(U)$, where $\chi_U$ is the characteristic function of a measurable subset $U\subseteq X$ of finite measure satisfying $0 < \mu(U) \leqslant \frac{1}{2}\mu(X)$.
For the ``only if'' implication, the lack of spectral gap is exactly the assertion of~\cite{schmidtSE}*{Proposition 2.3 (1)}  for $p=2$. Therefore, by~\cite{schmidtSE}*{Proposition 2.3}, this is equivalent to~\cite{schmidtSE}*{Proposition 2.3 (3)}, which coincides with \cite{schmidtSE}*{Lemma 2.1 (2)}. By a result of Rosenblatt in \cite{rosenblatt}, \cite{schmidtSE}*{Lemma 2.1 (2)} is equivalent to \cite{schmidtSE}*{Lemma 2.1 (1)}, which immediately implies the lack of measure expansion.

Complete proofs of Proposition \ref{prop:meSG} can be found in \cite{GMP}*{Lemma 5.2} and  \cite{FV}*{Proposition 7.3}.

\subsection{Ahlfors-regular measures}
The following definition relates measures and metrics. We denote by $B(x,r)$ the closed ball of radius $r$ centred at $x$.

\begin{defn}[Ahlfors-regular measure]
A measure $\mu$ on a metric space $X$ is called \emph{Ahlfors regular} if there exist $m>0$ and $C \geqslant 1$ such that for every $x\in X$ and $r \leqslant \diam X$, we have
\[
	C^{-1} r^m \leqslant \mu(B(x,r))\leqslant C r^m.
\]
\end{defn}

For a finite measure~$\mu$ on a bounded metric space~$X$, it suffices to check the above condition for all $r<\eps$ for a fixed $0 < \eps \leqslant \diam X$. The Lebesgue measure on the plane $\R^2$ is Ahlfors regular with $m=2$. Therefore, the measure on the torus $\T^2$ inherited from $[0,1]^2 \subseteq \R^2$ is also Ahlfors regular with $m=2$.

\section{Lifting measure expansion to extensions}
In this section, we prove a general result that allows to lift
expansion
in measure to extensions of dynamical systems. If $\Sigma' \to \Sigma$ is an $m$-sheeted branched covering of a surface $\Sigma$, then
the normalised measures $\mu$ on $\Sigma'$ and $\lambda$ on $\Sigma$ satisfy \eqref{measures-estimates} below, so our result applies  and we can lift measure expansion from $\Sigma$ to $\Sigma'$.
In \cref{subsec:actionsonsurfaces}, we will use this in the case $\Sigma=\T^2$.

\begin{prop}
\label{schmidt trick}
Let $(X,\mu)$ and $(Y,\lambda)$ be probability spaces with measure-preserving actions of a group $G=\langle S \rangle$. Suppose that there is a $G$-equivariant measurable map $\rho\colon X\to Y$ and a constant $m \in \N$ such that for every measurable $U\subseteq X$, its image $\rho(U)\subseteq Y$ is measurable and we have
\begin{equation}\label{measures-estimates}
	\frac{1}{m} \lambda\big( \rho(U) \big) \leqslant \mu(U) \leqslant \lambda\big( \rho(U) \big).
\end{equation}
If the action  $G \acts Y$ is expanding in measure and the action $G \acts X$ is ergodic, then the action $G \acts X$ is expanding in measure.
\end{prop}
Since $\mu(X) = \lambda(Y)=1$, formula \eqref{measures-estimates} implies that $\mu(\rho^{-1}(V)) = \lambda(V)$ for every measurable set $V\subseteq Y$; in other words, $Y$ is a factor of $X$ with respect to the action of $G$. Indeed, suppose that there exists $V$ contradicting this claim. Then $\mu(\rho^{-1}(V)) < \lambda(V)$, and hence $\mu(\rho^{-1}(Y \smallsetminus V)) > \lambda(Y \smallsetminus V)$, which violates \eqref{measures-estimates}.

As mentioned above, we will only apply \cref{schmidt trick} to $m$-sheeted branched coverings  $\Sigma \to \T^2$. In fact, when $X$ and $Y$ are topological spaces, $\mu$~and $\lambda$ have full supports, and the factor map~$\rho$ is open, one can check that \eqref{measures-estimates} implies that all fibres have cardinality bounded by $m$.

\begin{proof}[Proof of \cref{schmidt trick}]
We proceed in two steps. First, we show that the action $G \acts X$ is strongly ergodic. Suppose this is not the case. It follows from \cite{schmidtSE} (see \cite{AbertElek2012}*{Lemma 2.1} for the explicit statement) that there exists a sequence ${(U_n)}_n$ of almost invariant subsets of $X$ with $\mu(U_n) = \frac {1}{2m}$. Since $\rho$ increases the measure by at most the factor $m$, it is immediate that the sequence $\big(\rho(U_n)\big)_n$ is also almost invariant. Moreover, for all $n$, we have $\lambda\big( \rho(U_n) \big) \in \left[\frac 1 {2m} , \frac 1 2\right]$. This contradicts the definition of expansion in measure for the action $G \acts Y$. Hence, $G \curvearrowright X$ is strongly ergodic.

Now we show that $G \curvearrowright X$ is expanding in measure. Suppose this is not the case. Then there exists a sequence ${(U_n)}_n$ of measurable subsets of $X$ with $0 < \mu(U_n) \leqslant \frac 1 2$ such that $\lim_n \frac{\mu( SU_n \smallsetminus U_n)}{\mu(U_n)}=0$. Strong ergodicity of $G \curvearrowright X$ tells us that $\lim_n \mu(U_n) \big( \mu(X) - \mu(U_n) \big) = 0$, and since $\mu(U_n) \leqslant \frac 1 2$, we conclude $\lim_n \mu(U_n) = 0$.

Let $c$ be the expansion constant associated with the measure expanding action $G \curvearrowright Y$ and the generating set $S$. For $n$ sufficiently large, we have $\mu(U_n)\leqslant \frac 1 {2m}$, so $\lambda(\rho(U_n))\leqslant \frac 1 2$, and therefore,
\begin{equation}\label{expansion-downstairs}
 \lambda\big( S \rho(U_n) \smallsetminus \rho(U_n) \big) \geqslant c \lambda\big(\rho(U_n)\big).
\end{equation}
On the other hand, since $\lim_n \frac{\mu( SU_n \smallsetminus U_n)}{\mu(U_n)}=0$, we may, without loss of generality, assume that for all $n$ sufficiently large,
\[
 \mu( SU_n \smallsetminus U_n) < \frac c m \mu(U_n).
\]
Hence, for all $n$ sufficiently large, we obtain
\begin{align*}
 \lambda\big( S \rho(U_n) \smallsetminus \rho(U_n) \big)
 &= \lambda\big( \rho(SU_n) \smallsetminus \rho(U_n) \big)\\
 &\leqslant \lambda\big( \rho(SU_n \smallsetminus U_n) \big)\\
 &\leqslant m\mu( SU_n \smallsetminus U_n)\\
 &< c \mu(U_n)\\
 &\leqslant c \lambda\big( \rho(U_n) \big),
\end{align*}
contradicting \eqref{expansion-downstairs}. Hence, $G \curvearrowright X$ is expanding in measure.
\end{proof}

\section{Measure-expanding actions on origami surfaces}

\subsection{An action of \texorpdfstring{$\GL_2(\mathbb{Z})$}{GL2(Z)} on the torus}\label{subsec:GL}

Let $\mathbb{T}^2=\mathbb{R}^2 / \mathbb{Z}^2$ be the $2$-torus, which we identify, as a measure space, with $[0,1)^2 \subseteq \R^2$ equipped with Lebesgue measure. The group $\GL_2(\mathbb{Z})$ naturally acts on $\R^2$ by left multiplication. This action descends to an action on $\mathbb{T}^2$, which is explicitly given by
\begin{eqnarray}\label{eq:T2}
	\begin{pmatrix} p & q\\ r & s \end{pmatrix} . \begin{pmatrix} x \\ y \end{pmatrix}: = \begin{pmatrix} px + qy \\ rx + sy \end{pmatrix}.
\end{eqnarray}
\begin{defn}[Generic point]
 A point $\begin{pmatrix} x \\ y \end{pmatrix} \in \R^2$ is called \emph{generic} if the set $\{1,x,y\}$ is linearly independent over $\QQ$. A point $\begin{pmatrix} x \\ y \end{pmatrix} \in \T^2$ is called \emph{generic} if one (or equivalently, every one) of its preimages in $\R^2$ is generic.
\end{defn}

The generic points form a set of measure $1$.

\begin{lem} \label{freeness on torus}
Let $\begin{pmatrix} x \\ y \end{pmatrix}\in \mathbb{T}^2 $ be a generic point. The orbit $\GL_2(\mathbb{Z}) . \begin{pmatrix} x \\ y \end{pmatrix}$ in $\mathbb{T}^2$ is free and consists solely of generic points.
\end{lem}
\begin{proof}
Let $\gamma = \begin{pmatrix} p & q\\ r & s \end{pmatrix} \in \GL_2(\mathbb{Z})$, and suppose that $\gamma. \begin{pmatrix} x \\ y \end{pmatrix} = \begin{pmatrix} x \\ y \end{pmatrix}$, i.e.
\[
	\begin{pmatrix} x \\ y \end{pmatrix} = \begin{pmatrix} px + qy \\ rx + sy \end{pmatrix}.
\]
Thus,
\[ x = px + qy  \quad \text{ and } \quad y = rx + sy .\]
By the assumption of genericity of $x$ and $y$, we conclude that
\[
	p = 1, \; q = 0 \quad \text{ and } \quad r = 0, \; s = 1,
\]
so $\gamma$ is the identity matrix, showing that the orbit of $\begin{pmatrix} x \\ y \end{pmatrix}$ is free.

Now we prove the genericity of $\begin{pmatrix} px + qy \\ rx + sy \end{pmatrix}$. Let $\mu$ and $\nu$ be rational numbers such that
\[
\mu(px+qy) + \nu(rx+sy) = 0
\]
in $\R/\Z$.
Then we have $0= \mu p+ \nu r = \mu q+\nu s$ by the genericity of
$\begin{pmatrix} x \\ y \end{pmatrix}$, and hence, $(\mu \;\; \nu)\gamma =(0 \;\; 0)$. Since the matrix $\gamma$ is invertible, this forces $\mu = \nu = 0$.
\end{proof}

\subsection{A measure-expanding action \texorpdfstring{$\mathbb{F}_2 \acts \mathbb{T}^2$}{of F\_2 on T\^{}2}}
Consider the measure-preserving action $\SL_2(\mathbb{Z}) \acts \mathbb{T}^2$  obtained by restricting the action $\GL_2(\mathbb{Z}) \acts \mathbb{T}^2$
discussed
in \cref{subsec:GL}. We consider subgroups of $\SL_2(\mathbb{Z})$ generated by special elements. For an integer $k \geqslant 1$, consider the elementary matrices
\begin{equation} \label{eq:generators}
	a_k =\begin{pmatrix} 1 & k \\ 0 & 1 \end{pmatrix} \quad \textrm{and} \quad b_k =\begin{pmatrix} 1 & 0 \\ k & 1 \end{pmatrix},
\end{equation}
with inverses
\[
	a_k^{-1} =\begin{pmatrix} 1 & -k \\ 0 & 1 \end{pmatrix} \quad \textrm{and} \quad b_k^{-1} =\begin{pmatrix} 1 & 0 \\ -k & 1 \end{pmatrix}.
\]

\begin{thm} \label{thm:spectralgaptorus}
	For every $k \in \mathbb{Z}_{\geqslant 1}$, the measure-preserving action $\langle a_k,b_k \rangle \curvearrowright \mathbb{T}^2$ is expanding in measure. Equivalently, the action has spectral gap.
\end{thm}
It is standard  that for $k\geqslant 2$, the group $\langle a_k,b_k \rangle$ generated by $a_k, b_k$ is a free group~$\mathbb{F}_2$. Thus, for $k \geqslant 2$,  \cref{thm:spectralgaptorus} gives measure-expanding actions $\mathbb{F}_2 \curvearrowright \mathbb{T}^2$.

\Cref{thm:spectralgaptorus} is well known. However, since the measure expansion of our actions on origami surfaces heavily relies on it, we present a proof, along the lines of Gabber--Galil \cite{GG}. They showed the result for $k=1$, but their proof works for all $k \geqslant 1$. Other values of $k$ and other generating sets were considered in \cites{JM, LL, Lee}. The value of $k$ will be important later.

First, note that the matrices $a_k$ and $b_k$ induce a measure-preserving action of $\langle a_k, b_k \rangle$ on $\Z^2 \smallsetminus \{(0,0)\}$, which we equip with the counting measure, denoted by $|\cdot|$.
\begin{lem} \label{lem:Z2expanding}
	For every $k \in \mathbb{Z}_{\geqslant 1}$ and every finite subset $A \subseteq \mathbb{Z}^2 \smallsetminus \{(0,0)\}$, we have
	\[
		\big\lvert a_kA \cup a_k^{-1}A \cup b_kA \cup b_k^{-1}A \big\rvert \geqslant 2\big\lvert A \big\rvert.
	\]
	In particular, the action of $\langle a_k,b_k \rangle$ on $\mathbb{Z}^2 \smallsetminus \{(0,0)\}$ is expanding in measure.
\end{lem}
For $k=1$, a detailed proof can be found in \cite{LL}. We give the computation for arbitrary $k \geqslant 1$ for completeness.
\begin{proof}[Proof of \cref{lem:Z2expanding}]
Let $Q_1\coloneqq \{ (p,q) \in \mathbb{Z}^2 \mid p > 0, q \geqslant 0 \}$, let $Q_2\coloneqq \{ (p,q) \in \mathbb{Z}^2 \mid p \leqslant 0, q > 0 \}$, let $Q_3\coloneqq \{ (p,q) \in \mathbb{Z}^2 \mid p < 0, q \leqslant 0 \}$, and let $Q_4 \coloneqq  \{ (p,q) \in \mathbb{Z}^2 \mid p \geqslant 0, q < 0 \}$. Together, the sets $Q_1$, $Q_2$, $Q_3$ and $Q_4$ form a partition of $\mathbb{Z}^2 \smallsetminus \{(0,0)\}$ and satisfy the following inclusions, which are straightforward to verify:
\begin{align*}
	&a_kQ_1,\;b_kQ_1 \subseteq Q_1, \qquad &a_k^{-1}Q_2,\;b_k^{-1}Q_2 \subseteq Q_2,\\
	&a_kQ_3,\;b_kQ_3 \subseteq Q_3, \qquad &a_k^{-1}Q_4,\;b_k^{-1}Q_4 \subseteq Q_4.
\end{align*}
Also, the following identities hold:
\begin{align*}
	&a_kQ_1 \cap b_kQ_1 = \emptyset, \qquad &a_k^{-1}Q_2 \cap b_k^{-1}Q_2 = \emptyset,\\
	&a_kQ_3 \cap b_kQ_3 = \emptyset, \qquad &a_k^{-1}Q_4 \cap b_k^{-1}Q_4 = \emptyset.
\end{align*}
Now, let $A$ be a finite subset of $\mathbb{Z}^2 \smallsetminus \{(0,0)\}$, and set $A_i \coloneqq  A \cap Q_i$ for $i=1,2,3,4$, so that $A = A_1 \sqcup A_2 \sqcup A_3 \sqcup A_4$. Using the above identities and the fact that the action is measure preserving, the result now follows from a direct computation:
\begin{align*}
	\big\lvert a_kA &\cup a_k^{-1}A \cup b_kA \cup b_k^{-1}A\big\rvert \\
	&\geqslant \big\lvert a_kA_1 \cup b_kA_1 \cup a_k^{-1}A_2 \cup b_k^{-1}A_2 \cup a_kA_3 \cup b_kA_3 \cup a_k^{-1}A_4 \cup b_k^{-1}A_4\big\rvert \\
		&= \big\lvert a_kA_1 \big\rvert + \big\lvert b_kA_1\big\rvert + \big\lvert a_k^{-1}A_2\big\rvert + \big\lvert b_k^{-1}A_2\big\rvert + \big\lvert a_kA_3\big\rvert + \big\lvert b_kA_3\big\rvert + \big\lvert a_k^{-1}A_4\big\rvert + \big\lvert b_k^{-1}A_4\big\rvert \\
	&= 2\left(\big\lvert A_1\big\rvert+\big\lvert A_2\big\rvert+ \big\lvert A_3\big\rvert+\big\lvert A_4\big\rvert\right) \\
	&= 2\big\lvert A\big\rvert. \qedhere
\end{align*}
\end{proof}
We now use the Fourier transform to transfer the measure expansion from  \cref{lem:Z2expanding} to $\mathbb{T}^2$. Let $\Gamma=(V,E)$ denote the graph with $V\coloneqq \Z^2 \smallsetminus \{(0,0)\}$, where each vertex $(x,y)$ is adjacent to the following vertices (some of them may coincide):
\[
	a_k\begin{pmatrix} x \\ y \end{pmatrix} = \begin{pmatrix} x + ky \\ y \end{pmatrix}, \quad b_k\begin{pmatrix} x \\ y \end{pmatrix} =\begin{pmatrix} x \\ kx + y \end{pmatrix}, \quad a_k^{-1}\begin{pmatrix} x \\ y \end{pmatrix}=\begin{pmatrix} x - ky \\ y \end{pmatrix}, \quad b_k^{-1}\begin{pmatrix} x \\ y \end{pmatrix}=\begin{pmatrix} x \\ -kx + y \end{pmatrix}.
\]
We use a simplification of the argument of Gabber--Galil from \cite{Lee}. The \emph{Rayleigh quotient} of a non-zero square-summable function $f \colon V \to \mathbb{C}$ is defined as
\[
	\mathcal{R}_{\Gamma}(f) \coloneqq  \frac{\sum_{\{v,w\} \in E} |f(v)-f(w)|^2}{\sum_{v \in V} |f(v)|^2}.
\]
By a version of the discrete Cheeger inequality \cite{Lee}*{Lemma 2.1}, for every countable graph $\Gamma=(V,E)$ with a uniform bound $D$ on the degrees and every function $f \in \ell^2(V)$, there exists a finite subset $A \subseteq \supp(f) \subseteq V$ such that
\begin{equation} \label{eq:discretecheeger}
	\frac{|\partial A|}{|A|} \leqslant \sqrt{2D\mathcal{R}_{\Gamma}(f)}.
\end{equation}
By \cref{lem:Z2expanding}, it follows that the left-hand side of \eqref{eq:discretecheeger} is strictly larger than $0$, uniformly over all finite subsets $A \subseteq V$.

\begin{proof}[Proof of \cref{thm:spectralgaptorus}]
Recall that we identify $\mathbb{T}^2 = \mathbb{R}^2 / \mathbb{Z}^2$ with $[0,1)^2$ with Lebesgue measure~$\mu$. For every $p,q \in \mathbb{Z}$, the function $\chi_{p,q} \colon \mathbb{T}^2 \to \mathbb{C}$ defined by $\chi_{p,q}(x,y)\coloneqq e^{2\pi i(px+qy)}$ defines a character of $\mathbb{T}^2$ and every character is of this form. Thus, the character group can be identified with $\mathbb{Z}^2$. Let $\mathcal{F} \colon L^2(\mathbb{T}^2) \to \ell^2(\mathbb{Z}^2),\, f \mapsto \hat{f}$ denote the Fourier transform, which is given by
\[
	\hat{f}(p,q) \coloneqq  \int_{\mathbb{T}^2} f(x,y) e^{-2\pi i(px+qy)} d\mu(x,y).
\]
The set $\{\chi_{p,q} \mid p,q \in \mathbb{Z}\}$ is an orthonormal basis for the Hilbert space $L^2(\mathbb{T}^2)$ and every $f \in L^2(\mathbb{T}^2)$ can be written as a linear combination $f = \sum_{p,q \in \mathbb{Z}} \hat{f}(p,q)\chi_{p,q}$, where the coefficients $\hat{f}(p,q)$ are given by the Fourier transform. The action of $\langle a_k,b_k \rangle$ on $\mathbb{T}^2$ satisfies the following relations:
\[
	\chi_{p,q} \circ a_k = \chi_{p,kp+q}, \qquad \chi_{p,q} \circ b_k = \chi_{p+kq,q}.
\]
It follows that
\begin{equation} \label{eq:fouriertransformab}
	\widehat{f \circ a_k}=\hat{f} \circ b_k^{-1}, \qquad \widehat{f \circ b_k}=\hat{f} \circ a_k^{-1}.
\end{equation}
These identities were essentially already observed in \cite{rosenblatt}; see also \cite{schmidt}. We can now compute the spectral gap; cf.~\cref{defn:spectralgap}:
\begin{align*}
	\min \left\{ \frac{ \| f \circ a_k - f \|_2 + \| f \circ b_k - f \|_2 + \| f \circ a_k^{-1} - f \|_2 + \| f \circ b_k^{-1} - f \|_2}{ \| f \|_2 } \; \bigg\vert \; f \in L^2_0(\mathbb{T}^2), \; f \neq 0 \right\}
\end{align*}
\begin{align*}
	&= \min \left\{ \frac{ \sum_{T\in\{a_k,\, b_k,\, a_k^{-1},\, b_k^{-1}\}} \| \widehat{f \circ T} - \hat{f} \|_2 }{ \| \hat{f} \|_2 } \; \bigg\vert \; \hat{f} \in \ell^2(\mathbb{Z}^2), \; \hat{f}(0,0)=0,\; \hat{f} \neq 0 \right\} \\
	&= \min \left\{ \frac{ \sum_{T\in\{a_k,\, b_k,\, a_k^{-1},\, b_k^{-1}\}} \| \hat{f} \circ T - \hat{f} \|_2 }{ \| \hat{f} \|_2 } \; \bigg\vert \; \hat{f} \in \ell^2(\mathbb{Z}^2), \; \hat{f}(0,0)=0, \; \hat{f} \neq 0 \right\} \\
	&\geqslant 2 \min \left\{ \sqrt{\mathcal{R}_{\Gamma}(\hat{f})} \; \bigg\vert \; \hat{f} \in \ell^2(\mathbb{Z}^2), \; \hat{f}(0,0)=0, \; \hat{f} \neq 0 \right\},
\end{align*}
where the factor of $2$ comes from double-counting the edges.
By \cref{lem:Z2expanding} and the version of the discrete Cheeger inequality recalled above, the right-hand side is strictly larger than $0$, showing that the action $\mathbb{F}_2 \curvearrowright \mathbb{T}^2$ is measure expanding.
\end{proof}

\subsection{Measure-expanding actions of \texorpdfstring{$\mathbb{F}_2$}{F\_2} on surfaces} \label{subsec:actionsonsurfaces}

We now use the action on the torus to construct a family of measure-expanding actions on origami surfaces of arbitrary genus $g \geqslant 1$.

We first recall the notion of an origami surface. Particular cases of origami surfaces were first studied in \cites{Thurston,Veech}; the name origami was introduced in \cite{Lochak05}.
\begin{defn}[Origami]
\label{defn:origami}
An \emph{origami datum} is a triple $(m, \sigma, \tau)$, where $m$ is a positive integer, and $\sigma$ and $\tau$ are elements of the symmetric group of rank $m$ such that the subgroup they generate acts transitively on $\{1,\dots,m\}$.

An \emph{origami surface} (or a \emph{square-tiled translation surface}) associated with the origami datum $(m, \sigma, \tau)$ is a topological space $\Sigma$ constructed as follows:
Let $\{P_1, \dots, P_m\}$ be a set of copies of the square $[0,1]^2$ together with homeomorphisms $\rho_i \colon P_i \to [0,1]^2$. The square $[0,1]^2$ has left, right, top, and bottom sides, defined in the obvious way, and every square $P_i$ inherits such sides via $\rho_i$. The space $\Sigma$ is obtained by gluing, for every $i \in \{1,\ldots,m\}$, the left side of $P_i$ to the right side of $P_{\sigma(i)}$, in such a way that precomposing this gluing with ${\rho_i}^{-1}$ and postcomposing with $\rho_{\sigma(i)}$ yields a map sending the left side of $[0,1]^2$ to its right side by translation. We perform analogous gluings identifying the top side of $P_i$ with the bottom side of $P_{\tau(i)}$ for every $i \in \{1,\ldots,m\}$.

We endow $\Sigma$ with the length metric induced from declaring the length of paths on the squares $P_i$ to be the pull-back of the usual length on $[0,1]^2$ via $\rho_i$.
\end{defn}
The maps $\rho_i$ combine to a single map $\rho \colon \Sigma \to \T^2$, which is a branched covering of the torus $\T^2$ that branches over at most one point. 

A quick Euler characteristic count using the observation that the valence of every vertex in an origami surface is a positive multiple of $4$ shows that the sphere $\Sp^2$ is not an origami surface.

By definition, our origami surfaces are connected, since we assume that the group generated by the permutations $\sigma$ and $\tau$ acts transitively on $\{1,\ldots,m\}$. In fact, topologically, $\Sigma$ is always a connected, closed, orientable surface. It also has a measure $\mu$ of total measure $m$, inherited on each of the squares $P_i$ from the Lebesgue measure on $[0,1]^2$.

We depict origami surfaces geometrically by drawing $m$ squares in the plane, glued according to the permutations $\sigma$ and $\tau$, and with opposite (outer) edges identified (according to $\sigma$ and $\tau$). An example is given in \cref{Lshape}.

\begin{figure}
 \includegraphics[scale=0.6]{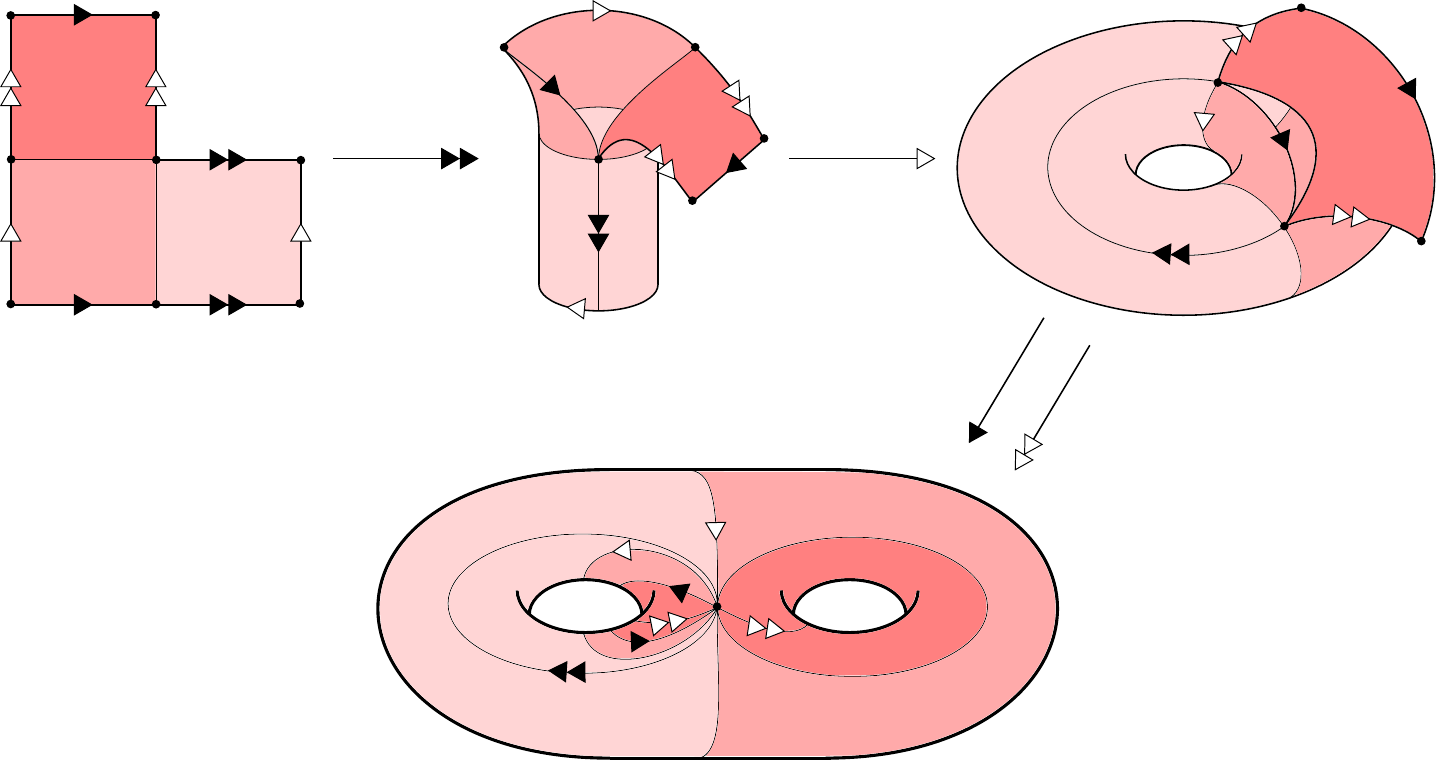}
  \caption{Constructing an origami surface of genus $2$}
 \label{Lshape}
\end{figure}

Since the measures $\mu$ on $\Sigma$ and $\lambda$ on $\T^2$ are both induced by the Lebesgue measure on $[0,1]^2$, we have $\lambda(\rho(U)) \leqslant \mu(U)$ for every measurable $U \subseteq \Sigma$. In the other direction, since $\Sigma$ consists of $m$ squares, we have $\lambda(\rho(U)) \geqslant \frac{1}{m}\mu(U)$. Hence, the map $\rho\colon \Sigma\to \T^2$ satisfies \eqref{measures-estimates} after normalising $\mu$ by $m$.

The image of a ball $B(y,r)$ in $\Sigma$ is the ball $B(\rho(y),r)$ of the same radius in $\T^2$. Hence, we have
\[ \lambda(B(\rho(y),r)) \leqslant \mu(B(y,r)) \leqslant m \lambda(B(\rho(y),r)).\]
Therefore, the measure $\mu$ is Ahlfors-regular because $\lambda$ is.

\begin{defn}[Arrangement maps]
Let $\Sigma$ be an origami surface, and let $P_i$ be a square of $\Sigma$. The \emph{horizontal arrangement map} associated with $P_i$ is the unique map $h_i \colon \R \times [0,1] \to \Sigma$  satisfying the following:
 \begin{itemize}
  \item for every $j \in \Z$, postcomposing the restriction of $h_i$ to $[j,j+1]\times[0,1]$ with $\rho$ coincides with the composition of the translation by $(-j,0)$ with the natural quotient map $[0,1]^2 \to \T^2$;
  \item the square $[0,1]^2$ is mapped to $P_i$.
 \end{itemize}
The \emph{vertical arrangement map} associated with $P_i$ is the map $v_i \colon [0,1] \times \R \to\Sigma$ defined analogously.
\end{defn}
Intuitively, the horizontal arrangement map $h_i$ is an arrangement of the squares of $\Sigma$, where we start with $P_i$, and then glue to it its left and right neighbours in $\Sigma$, as specified by the permutations of the origami datum; we then continue, and glue further left and right neighbours. We end up with a cyclically repeated pattern, whose period is at most $m$ and is a divisor of the order of $\sigma$.

Our goal is to construct a measure-expanding action of $\F_2$ on $\Sigma$. Let $k\geqslant 2$ be any multiple of the orders of $\sigma$ and $\tau$, and consider the matrices $a_k$ and $b_k$ from \eqref{eq:generators}. A particular choice we can make for $k$ is to set $k \coloneqq  m!$ (unless $m=1$), though later we will use the fact that we have flexibility in picking $k$. Recall that $a_k$ and $b_k$ generate a free group of rank $2$ that acts on $\R^2$ and on $\T^2$.

Observe that the matrices $a_k$ and $b_k$, when acting on $\R^2$, preserve the subsets $\R \times [0,1]$ and $[0,1]\times \R$, respectively. We use this fact to define the action of $a_k$ and $b_k$ on horizontal and vertical arrangements, respectively, and thus on $\Sigma$.

Consider a square $P_i$ of $\Sigma$, and let $h_i \colon \R \times [0,1] \to \Sigma$ denote the corresponding horizontal arrangement map. The action of $a_k$ on $P_i$ is obtained by first lifting $P_i$ to $[0,1]^2$ under $h_i^{-1}$, then applying the matrix $a_k$, and then projecting back to $\Sigma$ using $h_i$. This gives a well-defined action of $a_k$ on $\Sigma$: On horizontal edges this follows from the fact that $k$ is a multiple of the order of $\sigma$, and hence $a_k$ fixes the top and bottom sides of all squares pointwise; for vertical edges the claim follows from an easy computation. We define the action of $b_k$ analogously, using the vertical arrangement map.
It is clear that $\rho$ is now an $\F_2$-equivariant map, where the action of $\mathbb{F}_2$ on $\T^2$ is as in \cref{thm:spectralgaptorus}.

\begin{thm} \label{thm:origamimeasureexpanding}
 The action $\F_2 \curvearrowright \Sigma$ we have just described is a measure-preserving and measure-expanding action by Lipschitz homeomorphisms.
\end{thm}
\Cref{thm:origamimeasureexpanding} directly implies \cref{thm:F2}.
\begin{proof}
The fact that the action is a measure-preserving action by Lipschitz homeomorphisms is clear from the construction.

Observe that the map $\rho \colon \Sigma \to \T^2$ satisfies the assumptions of \cref{schmidt trick} (up to rescaling the measure on $\Sigma$). Since the action $\F_2 \curvearrowright \T^2$ is expanding in measure, we only need to verify that the action $\F_2 \curvearrowright \Sigma$ is ergodic.

To this end, let $U \subseteq \Sigma$ be an invariant measurable subset with $\mu(U) >0$.
Consider the set $L_0$ of all connected components of preimages under $\rho$ of all (horizontal) loops of the form $[0,1) \times \{x \} \subset [0,1)^2 = \T^2$ with $x \in [0,1] \smallsetminus \mathbb Q$. Note that $\bigcup L_0$ is a measurable subset of $\Sigma$ of full measure.
Note also that every $l \in L_0$ is a horizontal loop of integral length carrying it own, usual Lebesgue measure.

Since the measure on $\Sigma$ is obtained from product measures on the squares forming $\Sigma$, Fubini's theorem gives us a subset $L_1$ of $L_0$ such that $\bigcup L_1$ is again measurable and of full measure, and such that for every $l \in L_1$ the set $l \cap U$ is measurable in $l$.

The action of $a_k$ on every loop $l$ is by irrational rotation, and hence is ergodic. We conclude that for every $l \in L_1$, the subset $l \cap U$ is either null or of full measure, since $l$, and hence $l\cap U$, is $a_k$-invariant.
We repeat the argument for vertical loops and arrive at analogous conclusions with sets $L_0'$ and $L'_1$ of vertical loops.

Let $Q$ be a square of $\Sigma$ with $\mu(Q \cap U) > 0$. Identifying $Q = [0,1]^2$, we see that, up to measure zero, $U \cap Q = [0,1] \times R$ for a measurable subset $R \subseteq [0,1]$ of positive measure $\lambda(R)$.
Recall that $\bigcup L'_1$ is of full measure in $\Sigma$, and hence, in particular $L'_1$ contains almost all vertical loops passing through $Q$. Every such loop intersects
$[0,1]\times R$
in a subset of measure at least $\lambda(R)$.
This implies that $\mu(Q \cap U) =1$, using the ergodicity of the action of $b_k$ on the loops in $L'_0$, as before. To summarise: if $U$ intersects a square in a subset of positive measure, then the intersection is full in~$Q$.

Since $\mu(U) > 0$, there is at least one square $Q$ as above. But then the precise description of the action of $\F_2$, together with the fact that $\Sigma$ is connected, immediately allows us to conclude that $U$ intersects every square fully, and hence $\mu(U) = \mu(\Sigma) = m$. This finishes the proof.
\end{proof}

\begin{rem}
 \label{cor:locallyFree}
 It is immediate from \cref{freeness on torus} that the orbits of points $v \in \Sigma$ such that $\rho(v)$ is generic are free.
\end{rem}

\subsection{Staircases}
Let us introduce a particular family of origamis that will be important later.

\begin{defn}[Staircases]\label{staircaseDef}
 Given a positive integer $g$, the \emph{staircase of genus $g$}, denoted $Z_g$, is the origami surface associated to the origami datum $(2g-1,\sigma,\tau)$, where $\sigma = (1 \;\; 2)(3 \;\; 4)\cdots(2g-3 \;\; 2g-2)$ and $\tau = (2 \;\; 3)(4 \;\; 5)\cdots(2g-2 \;\; 2g-1)$.

 We endow the surface with the same action of $\F_2$ as before.
\end{defn}
As indicated by the notation, a staircase of genus $g$ gives rise to a surface of genus $g$.

\begin{figure}
 \includegraphics[scale=0.5]{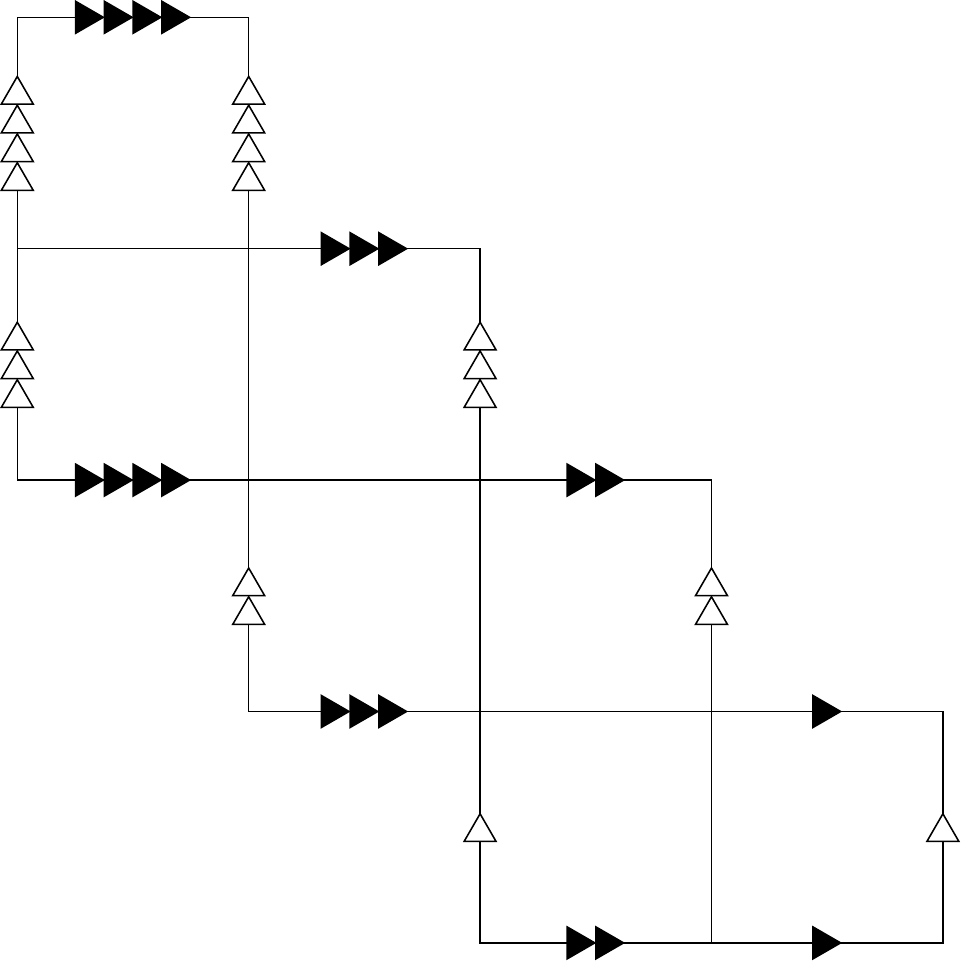}
 \caption{The staircase of genus 4}
 \label{staircase4}
\end{figure}

Observe that the orders of $\sigma$ and $\tau$ are both $2$, and we can take $k=2$ in the action of \cref{thm:origamimeasureexpanding}. Crucially, this is independent of $g$, and hence every $Z_g$ is a branched cover of $\T^2$ in an $\F_2$-equivariant way, with respect to a single action of $\F_2$ on $\T^2$, independent of $g$.

\section{Large-scale geometry, warped cones, and expanders}
\subsection{Uniform quasi-isometries and coarse embeddings} \label{subsec:ceqi}
We assume the reader is familiar with the following standard notions from large-scale geometry: quasi-isometry, quasi-geodesic space, bi-Lipschitz equivalence,  and the Rips complex $P_R(X)$ of a metric space $X$ for $R\geqslant 0$. We refer to the books by Bridson--Haefliger \cite{BH}, Roe \cite{roe:book}, and Nowak--Yu \cite{NY}, all of which are excellent sources on coarse geometry.

\begin{defn}[Uniform quasi-isometry]\label{def:qi}
Two families ${(X_n)}_n$ and ${(Y_n)}_n$ of metric spaces are {\it
uniformly quasi-isometric\/} if there exist functions $f_n\colon X_n\to Y_n$ and
constants $L\geqslant 1$ and $A\geqslant 0$ such that for all $n$ and all $x,x'\in X_n$ and $y\in Y_n$, we have
\begin{enumerate}[(i)]
  \item\label{QI-i} $L^{-1} d_{X_n}(x,x') - A \leqslant d_{Y_n}(f_n(x),f_n(x')) \leqslant L d_{X_n}(x,x') + A$,
    \item $d_{Y_n}(y, f_n(X_n))\leqslant A$.
\end{enumerate}
\end{defn}

Replacing the affine functions in \cref{def:qi} \eqref{QI-i} with arbitrary proper functions yields the following definition.

\begin{defn}[Uniform coarse equivalence]\label{def:ce}
Two families ${(X_n)}_n$ and ${(Y_n)}_n$ of metric spaces are \emph{uniformly coarsely equivalent}
if there exist functions $f_n\colon X_n\to Y_n$, two non-decreasing functions $\eta,\zeta \colon [0,\infty) \to [0,\infty)$ with $\lim_{r \to \infty}\eta(r)=\infty$,
and a constant $A\geqslant 0$ such that
for all
$x, x'\in X_n$ and $y\in Y_n$, we have
\begin{enumerate}[(i)]
  \item\label{CEq-i} $\eta \circ d_{X_n}(x,x') \leqslant d_{Y_n}(f_n(x),f_n(x')) \leqslant \zeta \circ d_{X_n}(x,x'),$
    \item\label{CEq-ii} $d_{Y_n}(y, f_n(X_n))\leqslant A$.
\end{enumerate}
\end{defn}

These definitions apply to individual spaces, which can be treated as constant sequences. We thus recover the usual definitions that two spaces $X$ and $Y$ are quasi-isometric or that they are {coarsely equivalent}, respectively.

If condition (\ref*{CEq-i})  is satisfied, but not necessarily (\ref*{CEq-ii}), we talk about embeddings. That is, applied to individual spaces $X$ and $Y$, \cref{def:qi} \eqref{QI-i} defines a \emph{quasi-isometric embedding} $f\colon X\to Y$, and \cref{def:ce} \eqref{CEq-i} defines a \emph{coarse embedding} $f\colon X\to Y$; if only the second inequality of \cref{def:ce} \eqref{CEq-i}  is satisfied, it defines a
\emph{bornologous} map $f\colon X\to Y$.

\begin{rem}\label{rem:cquasi}
If two families are uniformly quasi-isometric, then they are uniformly coarsely equivalent. Conversely, every coarse equivalence between quasi-geodesic metric spaces is a quasi-isometry~\cite{Gr:As}*{Section 0.2.D}; see e.g.~\cite{NY}*{Corollary~1.4.14} for a proof, which also holds  for families.
\end{rem}

\begin{conv}\label{conv:nounif}
Every map (in particular, every constant map) between bounded metric spaces is a quasi-isometry (for appropriate constants $L\geqslant 1, A\geqslant 0$), and in particular a coarse embedding. Accordingly,
applied to families ${(X_n)}_n$ and ${(Y_n)}_n$ of bounded metric spaces, the specification ``uniform'' from \cref{def:ce,def:qi} is often automatically assumed and omitted for brevity. This is the case in the formulations of our results in \cref{sec:introduction}; see also \cref{rem:single}.
\end{conv}

\begin{defn}[Close maps]
Let $f_1,f_2 \colon X \to Y$ be maps between metric spaces $(X, d_X)$ and $(Y, d_Y)$. If there is $C > 0$ such that $d_Y(f_1(x),f_2(x)) \leqslant C$ for all $x \in X$, then $f_1$ and $f_2$ are called $C$-\emph{close}.
\end{defn}

A quasi-isometric (resp.~coarse) embedding $f \colon X \to Y$ is a quasi-isometry (resp.~coarse equivalence) if and only if there is a quasi-isometric (resp.~coarse) embedding $f' \colon Y \to X$ such that $f' \circ f$ and $f \circ f'$ are close to the identity maps on $X$ and on $Y$, respectively; cf.~\cite{K}.

Throughout the article, we consider quasi-geodesic metric spaces, and we study them up to quasi-isometry. The following known result allows us to focus on graphs. We add the proof for completeness.
\begin{lem}\label{lem:qg}
A metric space is quasi-geodesic if and only if it is quasi-isometric to a graph.
\end{lem}
\begin{proof}
By definition, a metric space $X$ is quasi-geodesic if there exist $L \geqslant 1$ and ${A \geqslant 0}$ such that for every $x,x'\in X$, there exists an $(L,A)$-quasi-isometric embedding
$\gamma\colon [0, d(x,x')] \to X$ such that $\gamma(0)=x$ and $\gamma(d(x,x'))=x'$. A straightforward adaptation of the proof of \cite{CDP}*{Ch.~3, Lemma 2.8} to quasi-geodesic spaces yields that the canonical embedding $f\colon X \hookrightarrow P_{L+A}^{(1)}(X)$ is an $\left({L+A},1\right)$-quasi-isometric embedding, where $P_{L+A}^{(1)}(X)$ is the 1-skeleton of the Rips complex $P_{L+A}(X)$. Indeed, take $k\in\N$ such that $k-1<d_X(x,x')\leqslant k$.  By the quasi-geodesicity, there exist $a_0, a_1, \ldots, a_{k-1}, a_k \in X$ such that $a_0=x,\, a_k=x',$ and $d_X(a_i, a_{i+1})\leqslant L((i+1)-i)+A=L+A$. Hence, $P_{L+A}^{(1)}(X)$ is connected and $d_{P_{L+A}^{(1)}(X)}(f(x), f(x'))\leqslant k< d_X(x,x')+1.$ For the lower bound, suppose that $d_{P_{L+A}^{(1)}(X)}(f(x), f(x'))= l$. Then, by definition of the Rips complex, there exist $b_0, \ldots, b_l \in X$ such that $b_0=x,\, b_l=x',$ and $d_X(b_i, b_{i+1})\leqslant L+A$, whence $(L+A)^{-1}d_X(x,x') \leqslant d_{P_{L+A}^{(1)}(X)}(f(x), f(x')).$ Since $X$ is $\frac 1 2$-dense in $P_{L+A}^{(1)}(X)$, the above inequalities show that $f\colon X \hookrightarrow P_{L+A}^{(1)}(X)$ is an $\left(L+A, 1\right)$-quasi-isometry, from which the ``only if'' implication of the lemma follows.

Conversely, let $f\colon V\to X$ be an $(L,A)$-quasi-isometry, where $V$ is a graph. The image of a geodesic between any $v, v'\in V$ is an $(L,A)$-quasi-isometric embedding of $[0, d_V(v,v')]$. Since $f(V)$ is $A$-dense in $X$, the ``if'' implication follows.
\end{proof}

We encounter non-discrete metric spaces, for which we use the following general definition of bounded geometry.

\begin{defn}[Bounded geometry]\label{bddGeom}
A metric space $X$ has \emph{bounded geometry} if there exists a constant $C\geqslant 0$ such that for every $D>0$ there is $N_D\in \N$ such that every $C$-separated subset of $X$ of diameter at most $ D$ has cardinality at most $N_D$.
\end{defn}

To the knowledge of the authors, this notion of bounded geometry for metric spaces goes back to \cite{K}; the special case of Riemannian manifolds was already covered in \cite{Gr:HypMan}.

\begin{lem}\label{lem:bg}
A quasi-geodesic metric space has bounded geometry if and only if it is quasi-isometric to a graph with a uniform bound on the degrees of its vertices.
\end{lem}
\begin{proof}
The proof is routine.
For the ``if'' implication, note that a bounded-degree graph $V$ has bounded geometry with $C=1$. Let $X$ be a metric space. If $f\colon X\to V$ is an $(L,A)$-quasi-isometry and $T\subseteq X$ is a $C$-separated subset of diameter $D >0$, then $f(T)\subseteq V$ is a $(C/L-A)$-separated subset of diameter at most $LD+A>0$. We take $C> L(A+1)$. We see that $f$ is injective on $T$ and $|T|=|f(T)|\leqslant N_{LD+A}$, where $R\mapsto N_R$ is the function accompanying $C=1$ in \cref{bddGeom} for $V$. Hence, $X$ has bounded geometry.

For the ``only if'' implication, let $X$ be a quasi-geodesic space with bounded geometry. Let $X'$ be a maximal $C$-separated subset of~$X$ for the constant $C$ from \cref{bddGeom}. Then $X'$ is also quasi-geodesic. In the proof of \cref{lem:qg}, we have shown that the 1-skeleton $P_{L+A}^{(1)}(X')$ of the Rips complex of $X'$ is quasi-isometric to~$X'$, and hence to~$X$, where $L\geqslant 1, A\geqslant 0$ are the quasi-geodesicity constants of~$X'$. The maximal degree of the vertices of $P_{L+A}^{(1)}(X')$ is bounded above by  $N_{2(L+A)}(X)$, where $N$ is the function from the definition of bounded geometry of~$X$, since the ball of radius $1$ is of diameter at most $2$.
\end{proof}

\subsection{Warped cones} The following notion is useful for the understanding of the dynamics of group actions from the coarse geometry point of view. It was introduced by Roe in~\cite{roe:foli} for foliations and in~\cite{roe} for group actions.

\begin{defn}[Warped cone] \label{defn:warpedcone}
Let $(Y,d)$ be a compact metric space and $G\acts Y$ a group action by homeomorphisms. For every $t\in (0,\infty)$, let $d_G$ be the largest metric on $\{t\} \times Y$ satisfying
\begin{equation}
d_G\big((t,y),\, (t,y')\big) \leqslant t\cdot d(y,y') \quad \text{and} \quad d_G\big((t,y),\, (t, s.y)\big) \leqslant 1
\end{equation}
for every generator $s \in S$.

The sequence of metric spaces $\cO_G Y \coloneqq  (\{t\}\times Y, d_G)_{t\in \N}$ is called the \emph{warped cone} over the action $G\acts Y$.
\end{defn}

In~\cite{roe}, the warped cone is defined as an appropriate union over $t\in (0,\infty)$ of the levels $\{t\}\times Y$. Since the diameter of these levels tends to zero as $t \to 0$ and tends to infinity as $t \to \infty$ (unless $Y$ is a finite set with a transitive action), the resulting space indeed resembles a cone.

It easily follows from \cref{defn:warpedcone} that $d_G((t,y),\, (t,y')) < 1$ if and only if $d(y,y')<1/t$, and then $d_G((t,y),\, (t,y')) = t\cdot d(y,y')$, so both metrics induce the same topology. Hence, $\cO_G Y$ is not an interesting object from the topological perspective. However, it does exhibit interesting large-scale geometric properties. For example, it is not hard to check that if $Y$ is a geodesic space (or bi-Lipschitz equivalent to such a space), then the family $\cO_G Y$ is quasi-geodesic with the quasi-geodesicity constants $L\geqslant 1$ and $A\geqslant 0$ uniform over $t\in \N$. By \cite{roe}, most naturally occurring warped cones have bounded geometry. In particular, this applies to warped cones over actions by Lipschitz homeomorphisms on surfaces, as considered in the present article. Hence, by \cref{lem:bg}, a typical warped cone $\cO_G Y$ is uniformly quasi-isometric to a sequence of finite, connected graphs with a uniform bound on the degrees of their vertices. The relevance of this observation was noticed in \cite{FV}.

\begin{defn}[$\cO$-graph]\label{def:wcgr}
Every sequence $(\Gamma^{\cO}_n)_n$ of finite, connected graphs $\Gamma^{\cO}_n$ with uniformly bounded degrees that is uniformly quasi-isometric to a warped cone $\cO_G Y$ is called an \emph{$\cO$-graph} associated with $\cO_G Y$. \end{defn}

The following is a slightly generalised version of \cite{FV}*{Theorem~B}.

\begin{thm}\label{VigolosTheorem}
Let $Y$ be a compact space equipped with a geodesic metric and an Ahlfors-regular measure, and let $G\acts Y$ be a measure-preserving action by Lipschitz homeomorphisms. The action $G\acts Y$ has spectral gap if and only if every
$\cO$-graph associated with
the warped cone $\cO_G Y$ is an expander.
\end{thm}
\begin{proof}
If $Y$ is a Riemannian manifold, the result follows from \cite{FV}*{Corollary~7.4}, since every Lipschitz homeomorphism whose inverse is also Lipschitz is a quasi-symmetric homeomorphism. In general, the result is a special case, for $X=\R$, of a more general result \cite{super-exp}*{Theorem~1.1}.
\end{proof}

Consequently, the $\cO$-graphs associated with these warped cones are a source of expanders, which is in fact very rich, as exemplified by the results of \cite{FV,FNvL,deLaatVigolo,super-exp}. In particular, \cref{cor:origamiE} follows immediately from \cref{thm:F2} by applying \cref{VigolosTheorem}.

\begin{rem}\label{rem:qi}
Being an expander is a quasi-isometry invariant: If ${(\Gamma_n)}_n$ and ${(\Gamma'_n)}_n$ are two sequences of uniformly bounded degree graphs that are uniformly quasi-isometric, then ${(\Gamma_n)}_n$ is an expander if and only if ${(\Gamma'_n)}_n$ is. A proof of this folklore fact can be obtained by  a straightforward adaptation of the proof of \cite{K}*{Lemma 4.2}. Hence, the word ``every'' in  \cref{VigolosTheorem} can, equivalently, be replaced with ``some''.
\end{rem}

\section{Selberg-type expanders and Margulis-type expanders}

\subsection{Coarse disjoint unions and box spaces}

\begin{defn}[Coarse disjoint union] The disjoint union $\bigsqcup_n X_n$ of bounded metric spaces $(X_n,d_n)$  is called a \emph{coarse disjoint union} if it is equipped with a metric $d$ such that
\begin{itemize}
\item $d$ restricts to the metric $d_n$ on each $X_n$;
\item $d(X_n, X_m)\to\infty$ as $n+m\to \infty$ for all $n\not=m$.
\end{itemize}
\end{defn}

Any two such metrics yield coarsely equivalent metric spaces with underlying set~$\bigsqcup_n X_n$, or, in other words, the same coarse structure on~$\bigsqcup_n X_n$~\cite{roe:book}. Accordingly, one typically talks about \emph{the} coarse disjoint union.

Let $G$ be a group generated by a finite set $S$. We denote by $|g|_S$ the word length of an element $g\in G$ defined by $S$ and by $d_S(g,h)\coloneqq |g^{-1}h|_S$ the distance between elements $g,h\in G$ induced by this word length; $B(1_G,R)$ denotes the closed ball of radius $R\geqslant 0$ centered at the identity: $B(1_G,R) = \{g\in  G\mid |g|_S\leqslant R \}.$

\begin{defn}[Box space]\label{def:boxspace} Let ${(G_n)}_n$ be a sequence of finite-index normal subgroups of $G$.
We equip each $G/G_n$ with the quotient metric induced by the map $G\twoheadrightarrow G/G_n$ (equivalently, with the word-length metric defined by the image $S_n$ of $S$ under this quotient map).
We assume that
\begin{equation}\label{box-formula}
	\forall R\in \N\, \exists n_0\in \N \textrm{ such that } B(1_G, R) \cap \bigcup_{n\geqslant n_0} G_n = \{1_G\}.
\end{equation}
Then the coarse disjoint union $\bigsqcup_n \cay(G/G_n, S_n)$ is called a \emph{box space of $G$ defined by the sequence $\left( \cay(G/G_n, S_n) \right)_n$}.
\end{defn}

\begin{rem} \Cref{def:boxspace} is unusual and more general than the usual definition. If ${(G_n)}_n$ is nested and $\bigcap_nG_n=\{1_G\}$ (as in the standard definition of a box space often used in the literature), then $\bigsqcup_n \cay(G/G_n, S_n)$ is a box space in the sense of \cref{def:boxspace}. In fact, when ${(G_n)}_n$ is nested, then the condition $\bigcap_nG_n=\{1_G\}$ is equivalent to condition \eqref{box-formula}. However, if we only assume $\bigcap_nG_n=\{1_G\}$, then $\bigsqcup_n \cay(G/G_n, S_n)$ is not necessarily a box space according to our definition. Indeed, let $G = \Z \oplus \Z_2$, let $S\subseteq G$ be the standard generating set, let $G_1 = \Z \oplus \{0\}$, and let $G_n = n\Z \oplus \Z_2$ for $n\geqslant 2$. Then $\bigcap_{n\geqslant 1} G_n = \{1_G\}$, but for every $n_0\in \N$, the union $\bigcup_{n\geqslant n_0} G_n$ contains $\Z_2$, so for every $R\geqslant 2$ it intersects the ball $B(1_G,R)$ non-trivially.

The choice of this general definition of box space is inspired by \cref{def:faithful} and \cref{lem:generalised-box} below.\end{rem}

\begin{rem}\label{rem:single}
Despite the fact that an expander is defined as a \emph{sequence of graphs} and a box space is defined as a \emph{single metric space} consisting of such graphs, we often disregard the difference; cf.~\cite{AT:relative}*{Page 318}. 
There are historical reasons why these two points of view are used, and we choose to keep both languages.
\end{rem}

The following is a variation of the definition of asymptotically faithful sequence of graphs from~\cite{WY}*{Definition 2.2}.

\begin{defn}[Asymptotically faithful sequence of maps]\label{def:faithful}
Let $q_n\colon X_n \to Y_n$ be a sequence of surjective 1-Lipschitz maps between metric spaces. The sequence ${(q_n)}_n$ is called \emph{asymptotically faithful} if for every $R > 0$, there exists $N_R \in \N$ such that for every $n\geqslant N_R$, every $y\in Y_n$ and every $x\in q_n^{-1}(y)$, the map $q_n$ restricted to the map on the balls $B(x, R)\to B(y,R)$ is an isometry.
\end{defn}

\begin{lem}\label{lem:generalised-box} Let $G$ be a group generated by a finite set $S$, and let ${(G_n)}_n$ be a sequence of finite-index normal subgroups. The sequence of quotient maps $\cay(G,S)\twoheadrightarrow \cay(G/G_n,S_n)$ is asymptotically faithful if and only if $\bigsqcup_n \cay(G/G_n,S_n)$ is a box space.
\end{lem}
\begin{proof}
For $n \in \N$, let $R_n$ be the maximal integer such that $B(1_G, R_n) \cap G_n = \{1_G \}$. Clearly, $\bigsqcup_n \cay(G/G_n,S_n)$ is a box space if and only if $\lim_{n\to\infty} R_n = \infty$.

The quotient map $q_n\colon \cay(G,S)\twoheadrightarrow \cay(G/G_n,S_n)$ is isometric at scale $r_n\coloneqq R_n/4$, i.e., $q_n$ is an isometry when restricted to the ball $B(1_G,r_n)$. Indeed, suppose for contradiction that there are $g,h\in B(1_G,r_n)$ such that $d_S(g,h) > d_{S_n}(q_n(g), q_n(h))$. This means that $|g^{-1}h|_S > |q_n(g^{-1}h)|_{S_n}$. Then there exists $k\in G$ such that $q_n(k) = q_n(g^{-1}h)$ and $|k|_S = |q_n(g^{-1}h)|_{S_n}$. Hence, $k^{-1}g^{-1}h$ is a non-trivial element of $\ker q_n = G_n$. However,
\[
	|k^{-1}g^{-1}h|_S\leqslant |k^{-1}|_S + |g^{-1}h| _S< 2|g^{-1}h|_S \leqslant 2(|g^{-1}|_S + |h|_S) \leqslant 4r_n = R_n,
\]
contradicting the assumption on the choice of $R_n$.

On the other hand, $q_n$ is not an isometry at scale $r'_n\coloneqq \lceil (R_n+1)/2 \rceil$. Indeed, by the maximality in the definition of $R_n$, there is $k \in G_n$ of length $R_n+1$. Then there are $g\in B(1_G, \lceil (R_n+1)/2 \rceil)$ and $h\in B(1_G, \lfloor (R_n+1)/2 \rfloor)$ such that $g^{-1}h = k$. Thus, $q_n(g) = q_n(h)$, and hence $q_n$ is not only non-isometric, but even non-injective on $B(1_G,r'_n)$.

Summing up, the largest scale $t_n$ at which $q_n$ is isometric belongs to the interval $[r_n,r'_n) = [R_n/4, \lceil (R_n+1)/2 \rceil)$, so $t_n$ converges to infinity if and only if $R_n$ does, which finishes the proof.
\end{proof}

\subsection{Selberg-type expanders}\label{sec:Stype}

One of the goals of this article is to distinguish several classes of expanders from the point of view of coarse geometry. First, we introduce the following new class.

\begin{defn}[Selberg-type expanders]\label{def:selberg}
An expander is called a \emph{Selberg-type expander} if it is the sequence of graphs defining a box space of a finitely generated discrete subgroup of $\PSL_2(\mathbb R)$ or  of $\SL_2(\mathbb R)$.\end{defn}

There are various known examples of Selberg-type expanders:
\begin{enumerate}
\item\label{ex:S} If $G\leqslant \SL_2(\mathbb Z)$ is a finite-index subgroup generated by a finite set $S$ and $G/G_n$ are congruence quotients for $n\in \mathbb N$, then $\left( \cay(G/G_n, S_n) \right)_n$ is an expander by Selberg's celebrated
$3/16$ theorem on the spectral gap of the Laplacian
on the hyperbolic surfaces $\mathbb H^2/G_n$ of finite
volume~\cite{selberg}
combined with the comparison principle of Brooks~\cite{Br} and Burger~\cite{Bu} that ensures the equivalence between continuous and discrete spectral gaps.

\item\label{Selberg2} Let $G\leqslant \SL_2(\mathbb Z)$ be a subgroup generated by a finite set $S$, and let $G/G_p$ be the congruence quotients for primes $p$. Then $\left( \cay(G/G_p, S_p) \right)_p$ is an expander if and only if $G$ is non-elementary~\cites{BG06,BG08}.

\item\label{Selberg3} Let $G\leqslant \SL_d(\mathbb Z)$ be a subgroup generated by a finite set $S$ that is Zariski dense in $\SL_d$, and let $G/G_n$ be the congruence quotients for $n\in \mathbb N$. Then $\left( \cay(G/G_n, S_n) \right)_n$ is an expander~\cite{BV}. For $d=2$, this result provides a great diversity of Selberg-type expanders. Note that \eqref{Selberg2} is a special case of this general result because a subgroup of $\SL_2(\mathbb Z)$ is Zariski dense in $\SL_2$  if and only if it is non-elementary.
We refer to \cites{BG08a,BGS10,V12} for different approaches to further generalisations of \eqref{Selberg2} which are special cases of \eqref{Selberg3}.

\item\label{SelbergRational}
Let $G\leqslant \SL_2(\mathbb Q)$ be a subgroup generated by a finite set $S$. Since there are only finitely many denominators involved in the entries of the elements of $S$,  the congruence quotients $G/G_q$ are defined for all $q$ avoiding finitely many primes as divisors. It follows from the main result of \cite{SV} that as $q\to\infty$ ranges over
such
square-free integers, the sequence
$\left( \cay(G/G_q, S_q) \right)_q$ is an expander, provided that the connected component of the identity in the Zariski closure of $G$ is perfect.

This last condition is satisfied by non-amenable subgroups $G$ as above, since it follows from \cite{vanderPutSinger2003}*{Theorem 4.29} that the Zariski closure of such a $G$ in $\SL_2(\C)$ is the whole of $\SL_2(\C)$.
\end{enumerate}

\begin{rem}
Selberg's celebrated result~\cite{selberg} on the congruence quotients of $\SL_2(\Z)$ naturally invites to generalisations. One way to do this is to consider congruence quotients of other groups, not only of $\SL_2(\Z)$ or of its subgroups. This leads to the following definition due to Lubotzky: An arithmetic lattice is said to have the \emph{Selberg property} if it has Property $(\tau)$ with respect to appropriately defined congruence subgroups~\cite{LubotzkyEigen}.

An alternative way of generalising Selberg's result is to allow a wider family of finite quotients, not only congruence quotients. Despite the fact that the Selberg-type expanders previously considered in the literature are often made of congruence quotients, it is this alternative way we pursue in \cref{def:selberg}. Therefore, it is all the more striking that the new expanders provided by \cref{cor:origamiE}  are all coarsely distinct, by \cref{differentFromSelberg}, from all expanders of Selberg type, even under our very broad interpretations of Selberg-type behaviour and of box spaces as given in Definitions \ref{def:selberg} and \ref{def:boxspace}, respectively.
\end{rem}

\subsection{Margulis-type expanders} \label{subsec:margulistype}

Let $G=\SL_2(\mathbb Z)\ltimes \mathbb Z^2$, where the semi-direct product is defined by the left multiplication action of $\SL_2(\mathbb Z)$ on $ \mathbb Z^2$.
The following four elements generate $G$:
\[T_1= \left(\begin{pmatrix} 1 & 0 \\ 0 & 1 \end{pmatrix},\begin{pmatrix} 1  \\ 0 \end{pmatrix}\right),\ T_2=\left(\begin{pmatrix} 1 & 0 \\ 0 & 1 \end{pmatrix},\begin{pmatrix} 0  \\ 1 \end{pmatrix}\right),\ T_3=\left(\begin{pmatrix} 1 & 1 \\ 0 & 1 \end{pmatrix}, \begin{pmatrix} 0  \\ 0 \end{pmatrix}\right),\ T_4=\left(\begin{pmatrix} 0 & -1 \\ 1 & 0 \end{pmatrix},\begin{pmatrix} 0  \\ 0 \end{pmatrix}\right).
\]
A natural action of $G$ on the discrete torus $\mathbb Z_n\times \mathbb Z_n$ is defined by:
\begin{eqnarray}\label{eq:semid}
\left(\begin{pmatrix} a & b \\ c & d \end{pmatrix},\begin{pmatrix} u  \\ v \end{pmatrix}\right). \begin{pmatrix} x  \\ y \end{pmatrix}\coloneqq \begin{pmatrix} ax+by \\ cx+ dy \end{pmatrix}+\begin{pmatrix} u  \\ v \end{pmatrix}.
\end{eqnarray}

\begin{defn}[The original Margulis expander]\label{def:origM} Let $V_n\coloneqq \mathbb Z_n\times \mathbb Z_n$ and let $V'_n$ be a disjoint copy of $V_n$.
The \emph{original} Margulis expander~\cite{Mar73}*{Section 3.2} is the sequence $(M_n)_{n}$ of 5-regular bipartite graphs $M_n$
with vertex set $V_n\sqcup V'_n$ and edge set defined by the transformations $1_G, T_1, T_2, T_3, T_4$ according to the action defined by~(\ref{eq:semid}): A vertex $\begin{pmatrix} x  \\ y \end{pmatrix}\in V_n$ is adjacent to $\begin{pmatrix} x  \\ y \end{pmatrix}, \begin{pmatrix} x+1  \\ y \end{pmatrix}, \begin{pmatrix} x  \\ y+1 \end{pmatrix}, \begin{pmatrix} x  +y\\ y \end{pmatrix}, \begin{pmatrix} -y  \\ x \end{pmatrix} \in V'_n$. There are
multi-edges
and the 5-regularity counts edges with multiplicity. See \cref{margulis graph}.
\end{defn}

\begin{figure}
	\includegraphics[scale=0.7]{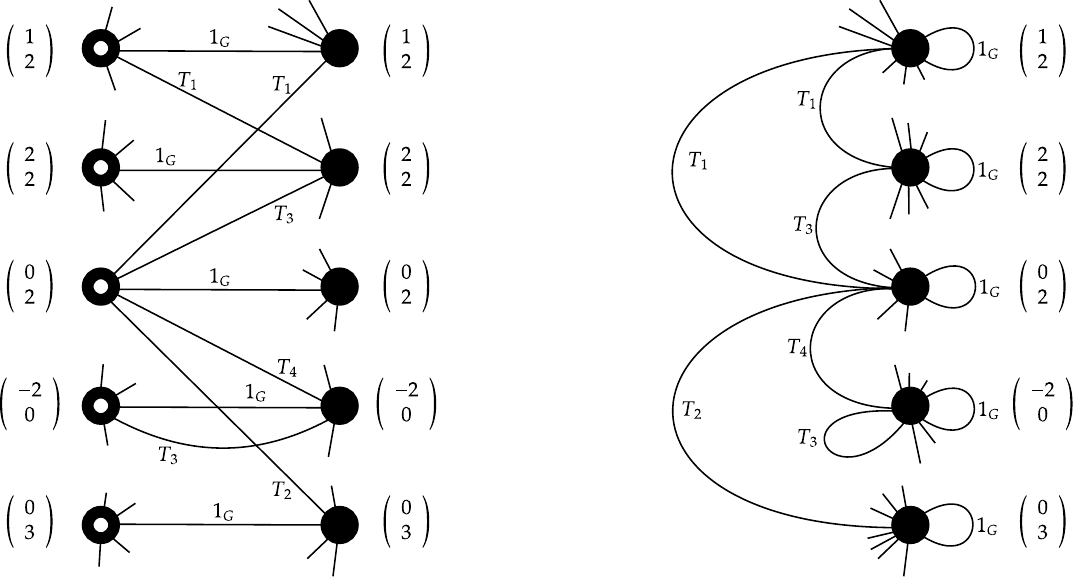}
	\caption{A local picture of the original Margulis expander $M_n$ (left) and the expander $M^\circ_n$ (right).}
		\label{margulis graph}
\end{figure}

Margulis shows that $(M_n)_{n}$ is a \emph{left} expander graph (in the natural bipartite sense of \cref{defn:expander}, where we require
$|\partial A|\geqslant (1+\varepsilon)|A|$ for every subset $A\subseteq V_n$ with $1\leqslant |A|\leqslant |V_n|/2$;
\cref{defn:expander}
also extends to graphs with loops and multiple edges). His proof also works to show the \emph{right} expansion of $(M_n)_{n}$ and the expansion of his other sequence $(\overline M_n)_{n}$ of 5-regular bipartite graphs~\cite{Mar73}*{Section 2.2}, defined as above but using the transformations $1_G, T_1, T_2, \overline T_3, T_4$ with $\overline T_3=T_4T_3^{-1}T_4^{-1}$. The Gabber--Galil expander~\cite{GG} is an analogous sequence $(L_n)_{n}$ of 5-regular bipartite graphs, defined by the transformations $1_G, T_3, \overline T_3, T_1T_3, T_2T_3$.  The next result is immediate from the definitions of these graphs and the fact that each of the sets $\{T_1, T_2, T_3, T_4\}$, $\{T_1, T_2, \overline T_3, T_4\}$, and $\{T_3, \overline T_3, T_1 T_3, T_2 T_3\}$ generates the entire semi-direct product $\SL_2(\mathbb Z)\ltimes \mathbb Z^2$.

\begin{lem}\label{lem:qiM1}
The Margulis expanders $(M_n)_{n}$ and $(\overline M_n)_{n}$ and the Gabber--Galil expander $(L_n)_{n}$ are uniformly quasi-isometric.
\end{lem}
By \cref{rem:qi}, this says in particular that the expansion of the original Margulis expander implies the expansion of the Gabber--Galil expander, and vice versa.

After choosing an arbitrary perfect matching $P_n$ in each graph $M_n$
and identifying $V_n$ with $V'_n$ along this perfect matching, we obtain a sequence $(M^\bullet_n)_n$ of 10-regular graphs $M^\bullet_n=M^\bullet_n(P_n)$
that is also an expander, as the boundaries of subsets in $M_n$ and in $M^\bullet_n$ are comparable: For a subset $A\subseteq V_n$ of the vertex set of $M^\bullet_n$ with $1\leqslant |A|\leqslant |V_n|/2$ we have
$|\partial A|\geqslant \varepsilon |A|$ whenever $|\partial X_A|\geqslant (1+\varepsilon)|X_A|$, where $X_A\subseteq V_n$ is a copy of $A$ in the vertex set of $M_n$. (If we removed the loops coming from $1_G$ from the graphs $(M^\bullet_n)_n$, or equivalently, if we collapsed the edges of the perfect matching chosen above, we would obtain a sequence of $8$-regular graphs.)

Alternatively, the Schreier graph for the action $\SL_2(\mathbb Z)\ltimes \mathbb Z^2\acts\mathbb Z_n\times \mathbb Z_n$ defined by~(\ref{eq:semid}), with respect to the generators $1_G, T_1, T_2, T_3, T_4$, yields
yet another sequence ${(M^\circ_n)}_n$ of 10-regular graphs that is an expander by Margulis's proof of relative property (T) for the pair $(\SL_2(\mathbb Z)\ltimes \mathbb Z^2, \mathbb Z^2)$~\cite{Mar73}*{Lemma 3.7}.

If $I_n$ is  the ``identity'' matching in $M_n$, i.e., a matching between each vertex $\begin{pmatrix} x  \\ y \end{pmatrix}\in V_n$ and its copy $\begin{pmatrix} x  \\ y \end{pmatrix}\in V'_n$, with a single
edge in $I_n$ for each such multi-edge in $M_n$, then $M^\bullet_n(I_n)$ coincides with $M^\circ_n$. See Figure \ref{margulis graph}.

\begin{lem}\label{lem:qiM2}
The original Margulis expander $(M_n)_{n}$ is uniformly quasi-isometric to the expander $(M^\circ_n)_n$ and the expanders $(M^\bullet_n(P_n))_n$ for arbitrary $P_n$.
\end{lem}
\begin{proof}
Since $M^\bullet_n(I_n)$ coincides with $M^\circ_n$, it suffices to show that $(M_n)_n$ is uniformly quasi-isometric with $(M^\bullet_n(P_n))_n$ for an arbitrary sequence of perfect matchings $(P_n)_n$. This follows immediately from the following claim.
\begin{cla*}\label{QIMatchings}
Let $\Gamma$ be a graph, $F\subseteq E(\Gamma)$ a matching, and $\Gamma'$ the quotient graph obtained by identifying vertices belonging to edges in $F$. Then the quotient map $q\colon (V(\Gamma), d_\Gamma)\to (V(\Gamma'), d_{\Gamma'})$, where $d_\Gamma$ and $d_{\Gamma'}$ denote the edge-length distances, is a $(2,1/2)$-quasi-isometry.
\end{cla*}
To prove the claim, first note that the quotient map~$q$ is surjective. Since it is a graph homomorphism, we have $d_{\Gamma'}(q(v), q(w)) \leqslant d_{\Gamma}(v,w)$ for every $v,w\in V(\Gamma)$.

To show the opposite estimate, let $v,w\in V(\Gamma)$ and let $q(v) = x_0', \ldots, x_n' = q(w)$ be a path in $\Gamma'$ realising the distance between $q(v)$ and $q(w)$, i.e.,\ $d_{\Gamma'}(q(v), q(w)) = n$. Since $q$ induces a surjective map between sets of edges, for every $i<n$, there are $x_i^2\in q^{-1}(x_i')$ and $x_{i+1}^1\in q^{-1}(x_{i+1}')$ such that $x_i^2$ and $x_{i+1}^1$ form an edge in $\Gamma$. Denote $x_0^1\defeq v$, $x_n^2 \defeq w$, and $p_0\defeq (x_0^1, x_0^2, x_1^1, x_1^2, x_2^1,\ldots, x_n^1, x_n^2)$.
By definition, $x_i^1, x_i^2 \in q^{-1}(x_i')$ for every $i\leqslant n$, so $x_i^1$ equals $x_i^2$ or they form an edge belonging to $F\subseteq E(\Gamma)$. Consider the sequence~$p$ obtained from~$p_0$ by omitting all vertices $x_i^2$ such that $x_i^1 = x_i^2$. Then $p$ is a path in $\Gamma$. Since its length as a sequence is at most $2n+2$, we get $d_\Gamma(v,w) \leqslant 2n+1 = 2d_{\Gamma'}(q(v), q(w)) + 1$. This finishes the proof of the claim, and hence of \cref{lem:qiM2}.
\end{proof}

The left multiplication action of $\SL_2(\mathbb Z)$ on $\mathbb R^2$ descends naturally to  an action by Lipschitz  homeomorphisms on the continuous $2$-torus $\T^2=\mathbb{R}^2 / \mathbb{Z}^2$; cf.~(\ref{eq:T2}). Let us consider the warped cone $\cO_{\SL_2(\mathbb Z)} \T^2$ over this action. The next result, together with the two preceding lemmata, shows that each of the expanders $(M_n)_{n}, (\overline M_n)_{n}, (L_n)_{n}, (M^\bullet_n(P_n))_n,$ and $(M^\circ_n)_n$ captures the dynamics of this action.

\begin{prop}\label{MargulisAsWarped}
The expander $(M^\circ_n)_n$ is uniformly quasi-isometric to the warped cone $\cO_{\SL_2(\mathbb Z)} \T^2$.
\end{prop}
\begin{proof}
The $n$th graph $M^\circ_n$ can be compared with the $n$th level $\{n\} \times \T^2$ via the obvious inclusion $(x,y)\mapsto (n, (\exp(2\pi i x/n), \exp(2\pi i y/n))$. Then the result follows, for example, from \cite{Scompletions}*{Corollary~2.3}.
\end{proof}

A variant of the above was also observed by other authors; see e.g.\ \cite{VigFund}*{Example 69}.

Thus,
even the oldest classical explicit constructions of expanders by Margulis and by Gabber--Galil are special cases of applying the warped cone construction. On the technical level, Margulis~\cite{Mar73} studies the action $\SL_2(\mathbb Z)\ltimes \mathbb Z^2\acts\Z^2$ and proves relative property (T), while
 to conclude that the warped cone $\cO_{\SL_2(\mathbb Z)} \T^2$ is uniformly quasi-isometric to an expander one needs that the action $\SL_2(\mathbb Z)\acts\T^2$ has spectral gap, which is indeed the approach of Gabber--Galil~\cite{GG}.

Inspired by this coarse geometric perspective, we introduce another new class of expanders.

\begin{defn}[Margulis-type expanders]\label{def:margulis}
An expander is called a \emph{Margulis-type expander} if it is uniformly quasi-isometric to the warped cone $\cO_G \Sigma$ over a measure-preserving action with a free orbit by Lipschitz homeomorphisms of a group $G$ on a closed surface $\Sigma$ of genus $g\geqslant 0$.
\end{defn}

In fact, the actions we  consider in what follows are, moreover, essentially free.

In genus $g=0$, examples of actions with spectral gap defining a Margulis-type expander are given by the action of the group $\F_2 = \langle a,b \rangle$ on the sphere $\Sp^2$, where $a$ and $b$ are two independent rotations represented by matrices in $\mathrm{SO}(3)$ with algebraic entries~\cites{drinfeld,gjs,BGSU2}.

A Margulis-type expander in genus $g=1$\label{classical-original} is called a \emph{classical Margulis-type expander}. This class encompasses both the original Margulis expander~\cite{Mar73} and the Gabber--Galil expander~\cite{GG}.
Moreover, this is potentially a huge class of expanders, since the group of area-preserving diffeomorphisms of the torus is an extremely rich object, studied intensively at least since the 1970s; see \cite{Banyaga}*{Chapter 5}.

For genus $g>1$, we use the terminology \emph{higher-genus Margulis-type expander}.

\begin{defn}[Origami expander] \label{def:origamiexpander}
	An \emph{origami expander} is an expander that is uniformly quasi-isometric to the warped cone $\cO_G \Sigma$ over an
	affine action with a free orbit of a group $G$ on an origami surface $\Sigma$.
\end{defn}

Margulis-type expanders and origami expanders are both $\cO$-graphs associated with warped cones over actions on surfaces. \Cref{def:margulis} requires the group to act by Lipschitz and measure-preserving homeomorphisms and \cref{def:origamiexpander} additionally requires the homeomorphisms to respect the affine structure of an origami surface. Hence, all origami expanders are Margulis-type expanders, and in particular the origami expanders from \cref{cor:origamiE} are  of Margulis type.

\section{Margulis-type expanders do not embed coarsely into Selberg-type expanders} \label{SelbergVSMargulis}

In this section, we show that any Margulis-type expander is coarsely non-equivalent to any Selberg-type expander. In fact, the former do not even admit coarse embeddings into the latter. We prove the following result, which immediately implies \cref{differentFromSelberg}.

\begin{thm} \label{differentiation}
There does not exist a coarse embedding from any Margulis-type expander into any Selberg-type expander.
\end{thm}

That the above statement holds for the original Margulis expander, one can argue as follows:
From~\cite{SWstraightening}*{Example~3.4}, it follows that the original Margulis expander \cite{Mar73} (see \cref{def:origM}) does not admit a fibred coarse embedding into a Hilbert space. We refer to \cite{CWY} for the definition of fibred coarse embedding. In contrast, the Selberg expander $(\cay (\SL_2(\Z / n\Z), S_n))_{n}$ admits such an embedding by \cite{CWW}*{Theorem~2.3} and \cref{lem:generalised-box} above, as it can be turned into a box space in the sense of \cref{def:boxspace} of a group with the Haagerup property. In the same way, since countable subgroups of $\PSL_2(\R)$ and  $\SL_2(\R)$ satisfy the Haagerup property~\cite{GHW}, all Selberg-type expanders admit a fibred coarse embedding into a Hilbert space. Therefore, the original Margulis expander does not admit a coarse embedding into any Selberg-type expander.

To prove \cref{differentiation} for arbitrary Margulis-type expanders, we apply a different strategy. We show that Margulis-type expanders have piecewise asymptotic dimension at least 2 (see \cref{piecewiseDef}), while the Selberg expander $(\cay (\SL_2(\Z / n\Z), S_n))_{n}$, and similarly, all expanders from items (\ref{ex:S}), (\ref{Selberg2}), and (\ref{Selberg3}) in \cref{sec:Stype} have piecewise asymptotic dimension 1~\cite{Spiecewise}*{Lemmata~6.5 and~6.6}.
 This is not true for some other Selberg-type expanders, in particular those obtained from the construction in item (\ref{SelbergRational}) of \cref{sec:Stype} applied to surface groups of genus larger than $1$  embedded into $\SL_2(\mathbb Q)$, as produced by Takeuchi~\cite{Takeuchi1971}.
 To deal with examples of this kind, we observed in \cref{lem:generalised-box} that such expanders locally resemble their mother group, i.e.\ the surface group. Since every Margulis-type expander coarsely contains arbitrarily large Euclidean 2-balls, \cref{differentiation} follows from the fact that the Euclidean plane does not coarsely embed into the hyperbolic plane (see~\cref{specialRH}), even though both have asymptotic dimension~2.

\smallskip

We now recall the notion of piecewise asymptotic dimension, which is an instance of piecewise versions of metric properties introduced in \cite{Spiecewise} in the case of asymptotic dimension; see e.g.\ \cites{NY, roe:book} for the definition of asymptotic dimension. A condition related to \cref{piecewiseDef} for $m=1$ is implicit in \cite{yamauchi}.

\begin{defn}[Piecewise asymptotic dimension] \label{piecewiseDef}
A sequence $(X_t)_{t}$ of metric spaces  has \emph{piecewise asymptotic dimension at most $m\in \N$} if for every $R\in \N$, there is $F(R)<\infty$ such that the family $\bigcup_R X_{F(R)+}^R$ of metric spaces has asymptotic dimension at most $m$, where for $S, F>0$, we denote
\[
	X_{F+}^S \defeq \{ A : A \subseteq X_t,\, t > F,\ \text{and } \diam A \leqslant S\}.
\]
We say that the \emph{piecewise asymptotic dimension of $(X_t)_t$ is infinite} if no such $m$ exists, and otherwise the smallest such integer $m$ is called the \emph{piecewise asymptotic dimension of~$(X_t)_t$}.
\end{defn}

\begin{prop}\label{hohoho} Let $\cO_{G}\Sigma$ be the warped cone over an action by Lipschitz homeomorphisms of a group $G$ on a Riemannian  surface $\Sigma$ with finitely many singularities, and assume that $\Sigma$ contains a free orbit disjoint from the singular set. Then there exist two non-decreasing functions $\eta,\rho \colon [0,\infty) \to [0,\infty)$ with $\lim_{r \to \infty}\eta(r)=\infty$ such that for every $R < \infty$ and for $t$ sufficiently large (depending on $R$), there is a $(\eta,\rho)$-coarse embedding of the ball $B(0,R)\subseteq \R^2$ into $\{t\}\times \Sigma \subseteq \cO_{G} \Sigma$.
\end{prop}
\begin{proof}
Let $x \in \Sigma$ be a point in
the
free orbit, and let $r$ be the injectivity radius at~$x$. Fix $0 < r'< r$. Then there is $L>1$ such that the exponential map $\R^2\supseteq B(0,r') \to B(x,r')\subseteq \Sigma$ and its inverse are Lipschitz with Lipschitz constant $L$.

Let $R < \infty$. Let $U \subseteq B(x,r')$ be a compact neighbourhood of $x$ such that $U \cap g U = \emptyset$ for every $g \in G \smallsetminus \{1\}$ with $|g|\leqslant 2R$.
The existence of such a neighbourhood follows from the Hausdorff property. Indeed, for every $g \neq 1$ with $|g| \leqslant 2R$, there are disjoint neighbourhoods $V_g$ of $x$ and $W_g$ of $gx$, so $\widetilde{V}_g \coloneqq  V_g \cap g^{-1}W_g$ satisfies $\widetilde{V}_g \cap g \widetilde{V}_g = \emptyset$. By intersecting the neighbourhoods $\widetilde{V}_g$ over $g$, we obtain a neighbourhood with the desired property, and we can further shrink it to guarantee that it is compact and contained in $B(x,r')$.

Denote $c \coloneqq  \min_{0< |g| \leqslant 2R} d(U, g U) > 0$. Now, let $t>0$ be sufficiently large that
\begin{itemize}
\item the set $U$ contains the ball $B(x,R/t)$, and
\item $ct > 2R$.
\end{itemize}
Let $U_t = \{t\} \times B(x, R/t)$ and consider the function $D_G \colon U_t\times U_t\to [0,\infty)$ given by
\begin{equation}\label{skewMetric}
	D_G (y,y') = \inf_{g \in G} \big( |g| + d(gy, y') \big),
\end{equation}
where $d$ denotes the metric on $\Sigma$ scaled by $t$.
We claim that $D_G$ equals the metric $d$. Indeed, first note that by considering $g=1_G$ in \eqref{skewMetric} we obtain $D_G(y,y') \leqslant |1_G| + d(1_G y, y') = d(y,y')\leqslant 2R$. Hence, $g \in G$ with $|g| > 2R$ can be neglected in \eqref{skewMetric}. For the opposite inequality, note that if $|g| \leqslant 2R$, then $|g| + d(g y, y') \geqslant |g| + ct > 2R > d(y,y')$ unless $g = 1_G$.

It follows from \cite{Scompletions}*{Proposition 2.1} that, if one extends $D_G$ to the metric on $\cO_G \Sigma$ given by the same formula \eqref{skewMetric}, then the identity map $(\cO_G \Sigma,d_G) \to (\cO_G \Sigma,D_G)$ is a coarse equivalence with control functions depending only on the Lipschitz constants of the generators of $G$. Since $(U_t,D_G) = (U_t, d)$ is bi-Lipschitz equivalent to the $R$-ball in $\R^2$ (again with the constants independent of $R$ and $F$), this finishes the proof.
\end{proof}

For an isometric action $G \curvearrowright \Sigma$, a more refined version of \cref{hohoho} was established in \cite{deLaatVigolo}*{Lemma 18}. For a free action, see \cite{SWstraightening}*{Proposition 3.10}.

\begin{cor} \label{cor:PWasdim}
Let $G$ and $\Sigma$ be as in \cref{hohoho}. Then, every unbounded family of levels of $\cO_{G}\Sigma$ has piecewise asymptotic dimension $\geqslant 2$. In particular, it does not admit a coarse embedding into any box space of a virtually free group.
\end{cor}
\begin{proof}
Clearly, $k = 2$ is the smallest value of $k$ such that the family of disks $\{B(0,r)\}_{r\in \N}$ in $\R^2$ has asymptotic dimension at most $k$ uniformly. Indeed, the upper bound is obvious and the lower bound follows from a limit argument; see e.g.\ \cite{Spiecewise}*{Proposition~6.8}.

It follows from \cref{hohoho} that for any choice of $F(R)<\infty$, the family $\bigcup_R X_{F(R)+}^R$ from \cref{piecewiseDef} for $X = \cO_{G} \Sigma$ contains the family of disks $\{B(0,r)\}_{r\in \N}$ coarsely, and hence the piecewise asymptotic dimension of $\cO_{G} \Sigma$ is at least the asymptotic dimension of $\{B(0,r)\}_{r\in \N}$, namely~$2$.

On the other hand, every infinite virtually free group $H$ has asymptotic dimension~$1$. Hence, by \cite{Spiecewise}*{Theorem~6.9}, every box space of $H$ has piecewise asymptotic dimension~$1$. As piecewise asymptotic dimension is non-decreasing under coarse embeddings~\cite{Spiecewise}*{Lemma~6.3}, the result follows.
\end{proof}

\begin{cor}\label{large-girth}  Let $(\Lambda_n)_n$ be an arbitrary sequence of graphs with ${\rm girth}\, \Lambda_n\to \infty$ as $n\to \infty$.  Let $G$ and $\Sigma$ be as in \cref{hohoho}.  Then there does not exist a coarse embedding of any unbounded family of levels of $\cO_{G}\Sigma$ into $(\Lambda_n)_n$. In particular, an $\cO$-graph associated with $\cO_{G}\Sigma$ is never a
large-girth expander.
\end{cor}
\begin{proof}
By \cite{yamauchi}*{Lemma 3.2}, every sequence of large-girth graphs has piecewise asymptotic dimension $1$, since balls of a fixed radius are eventually trees. The conclusion then follows by  \cref{cor:PWasdim}.
\end{proof}

\Cref{cor:PWasdim} proves \cref{differentiation} for those Selberg-type expanders that come from virtually free discrete subgroups of $\PSL_2(\R)$ or $\SL_2(\R)$. As mentioned before, using the following theorem, we can prove the same result for all Selberg-type expanders.

\begin{thm}\label{specialRH} The Euclidean plane $\R^2$ does not admit a coarse embedding into the hyperbolic plane $\HHH^2$.
\end{thm}

The key step in proving the theorem lies in showing that a putative coarse embedding $\R^2 \to \HHH^2$ would be coarsely surjective; this is sufficient, since for geodesic spaces coarsely surjective coarse embeddings are quasi-isometries. To prove coarse surjectivity one can adapt the proof of \cite{EskinFarb1997}*{Lemma 8.2} (where the underlying idea is credited to Mess) or the proof of \cite{KL}*{Theorem 3.8}.
\cref{specialRH} is also implied by \cite{DKbook}*{Theorem 9.69 and Corollary 9.70} and \cite{Li}*{Corollary 4.49}.

\begin{cor}\label{MargulisVsSurface} Let $G$ and $\Sigma$ be as in \cref{hohoho}, and $\cO_{G}\Sigma$ be the corresponding warped cone. An unbounded family of levels of $\cO_{G}\Sigma$ does not admit a coarse embedding into any box space of a finitely generated discrete subgroup of $\PSL_2(\mathbb R)$ or $\SL_2(\mathbb R)$.
\end{cor}

\begin{proof} Assume that a sequence of levels $(\{t_i\}\times \Sigma)_i$ admits a coarse embedding into a box space $\bigsqcup_i H/H_i$ for some finitely generated group $H$. By \cref{hohoho}, we have the following objects: a sequence $R_i\geqslant 0$ with $\lim_{i\to\infty} R_i = \infty$; functions $\eta$ and $\rho$ as in \cref{def:ce}; $(\eta,\rho)$-coarse embeddings $\iota_i^j \colon B(0,R_i)\to \{t_j\}\times \Sigma$ for all $i\in \N$ and $j\geqslant i$, where $B(0,R_i)$ denotes the disk of radius $R_i$.

Now, we can use the following standard argument from \cite{KV}*{Proposition 16} to obtain a coarse embedding $\R^2 \to H$: Since the sequence of maps $H\to H/H_i$ is asymptotically faithful (\cref{lem:generalised-box}), for every $i$ we can pick $j$ sufficiently large for the image of the composition
\[ B(0,R_i) \xrightarrow{\iota_i^{j}} \{t_{j}\}\times \Sigma \to H/H_{j}\]
to be contained in a ball of a sufficiently small radius $r$ such that $H\to H/H_{j}$ is isometric at the scale~$r$. We then compose the map further with the inverse of a restriction of the quotient map $H\to H/H_{j}$. This yields a coarse embedding $\iota_{R_i}\colon B(0,R_i)\to H$, and
may easily arrange to have $\iota_{R_i}(0)=1_H$. By taking a limit along an ultrafilter, we get a coarse embedding $\iota\colon \R^2\to H$.

Now, suppose that $H$ is a discrete subgroup of $\PSL_2(\mathbb R)$ or $\SL_2(\mathbb R)$. In particular, $H$ admits a coarse embedding into $\HHH^2$ obtained by choosing an arbitrary point $b$ in $\HHH^2$ and mapping $h \mapsto h.b$, where the action is inherited from the isometric action of $\PSL_2(\mathbb R)$ on $\HHH^2$. Consequently, there can be no coarse embedding of a sequence of levels $(\{t_i\}\times \Sigma)_i$ into $(H/H_i)_i$ as it would yield a coarse embedding $\R^2\to H\to \HHH^2$, contradicting \cref{specialRH}.
\end{proof}

\begin{proof}[Proof of \cref{differentiation}]
By definition, every Selberg-type expander can be turned into a box space of a finitely generated discrete subgroup of $\PSL_2(\mathbb R)$ or $\SL_2(\mathbb R)$. Hence, the claim follows from \cref{MargulisVsSurface} because the actions that we consider satisfy its assumptions by \cref{cor:locallyFree}.
\end{proof}

\section{Staircase expanders are not coarsely equivalent to box spaces}\label{StaircaseNotQIBox}

In \cref{SelbergVSMargulis}, we have shown that Margulis-type expanders do not admit a coarse embedding into  Selberg-type expanders. In particular, a Margulis-type expander cannot be coarsely equivalent (equivalently, quasi-isometric) to a Selberg-type expander. In this section, we strenghten this latter statement from Selberg-type expanders to arbitrary box spaces in the case of staircases. We do so by proving that the \emph{discrete fundamental groups} of levels of the associated warped cones are trivial. For box spaces (also according to our more general \cref{def:boxspace}), discrete fundamental groups are infinite \cite{DK}, which excludes the existence of a quasi-isometry between our warped cones and any box space of any infinite group.

Furthermore, a rich source of actions with spectral gap are translation actions of subgroups of compact Lie groups on the ambient compact Lie group. The spectral gap may simply be a consequence of Property (T) of the subgroup, but it can also occur rather generically for free subgroups \cites{BGSU2,BG2,BdS}. For subgroup actions, the discrete fundamental group of a level of the warped cone is infinite, as it contains the acting group or a group that surjects onto it \cites{VigFund,FNvL}. Hence, our warped cones are also not coarsely equivalent to any warped cone over an action from this large class.

The only examples of expanders with trivial discrete fundamental groups that were known before come from spectral-gap actions by rotations on even-dimensional spheres \cite{VigFund}. A novel feature of our examples is that the triviality of the discrete fundamental group occurs for warped cones over actions on manifolds whose fundamental group is infinite and even non-amenable.

\smallskip

We now recall the definition of the discrete fundamental group from \cite{BCW}; see also \cite{DK}. For $r > 0$, a sequence $(x_0, x_1, \ldots, x_n)$ of points in a metric space $(X,d)$ is called an \emph{$r$-path} if $d(x_i, x_{i+1})\leqslant r$ for $i=0,\ldots, n-1$. Fix a basepoint $x_{*} \in X$. An $r$-path $(x_0,x_1,\ldots,x_n)$ with $x_{*}=x_0=x_n$ is called an \emph{$r$-loop} based at $x_{*}$. We identify an $r$-loop $(x_0, \ldots, x_n)$ based at $x_{*}$ with its trivial extensions $(x_{*},x_0, \ldots,x_n)$ and $(x_0,\ldots,x_n,x_{*})$. Furthermore, we identify two $r$-loops $(x_0, \ldots, x_n)$ and $(y_0, \ldots, y_n)$ (of the same length) based at $x_{*}$ if they are $r$-close, namely $d(x_i,y_i)\leqslant r$ for all $i=1,\ldots,n-1$. The resulting equivalence relation on $r$-loops based at $x_{*}$ is called an \emph{$r$-homotopy}. Two $r$-loops $x = (x_0, \ldots, x_n)$ and $y = (y_0, \ldots, y_m)$ based at $x_{*}$ can be concatenated, yielding the $r$-loop $(x_0, \ldots, x_n, y_1, \ldots, y_m)$.
\begin{defn}[Discrete fundamental group]
The set of equivalence classes of $r$-loops based at $x_{*}$ under $r$-homotopy becomes a group under concatenation, which is called the \emph{discrete fundamental group} of $X$ at scale $r$ based at $x_{*}$ and denoted by $\pi_{1,r}(X,x_{*})$.
\end{defn}

A metric space $X$ is called \emph{$1$-geodesic} if for every $x,y\in X$, there is a $1$-path $(x_0, \ldots, x_n)$ with $x_0 = x$ and $x_n=y$ such that $d(x,y) = \sum_{i=0}^{n-1}d(x_i, x_{i+1})$. The usual metric on a connected graph is obviously 1-geodesic, and similarly, the warped metric on a level of a warped cone over a geodesic space is 1-geodesic by \cite{roe}*{Proposition~1.6}. For $r>1$, the map $\pi_{1,1}(X,x_{*}) \to \pi_{1,r}(X,x_{*})$ induced by the identity map on $X$ is surjective whenever $X$ is 1-geodesic, and hence, $\pi_{1,r}(X,x_{*}) \to \pi_{1,r'}(X,x_{*})$ is surjective whenever $1\leqslant r < r'$. Thus, increasing the scale $r$ amounts to adding new relations rather than new generators to the discrete fundamental group. A similar surjectivity result holds for maps of discrete fundamental groups induced by quasi-isometries:
\begin{lem}[\cite{VigFund}*{Lemma 21.(ii)}]\label{surjectivityOfQI}
Let $X,Y$ be $1$-geodesic spaces and $x_{*} \in X$. If $f\colon X\to Y$ is a quasi-isometry, then there exist constants $L,A\geqslant 0$ such that the induced map $f_{*}\colon\pi_{1,r}(X,x_{*})\to \pi_{1,r'}(Y,f(x_{*}))$ is surjective whenever $r\geqslant L+A$ and $r'\geqslant Lr+A$.
\end{lem}

Given $r>0$ and a continuous loop $\omega \colon \mathbb S^1 \to X$ based at $x_{*} \in X$, uniform continuity provides $n\in \N$ such that for all $\theta, \theta'\in [0,1]$ with $|\theta-\theta'|\leqslant \frac 1 {n}$, we have $d(\omega(\theta), \omega(\theta'))\leqslant r$. Thus, the sequence $(\omega(\frac i {n}))_{i=0}^{n}$ is an $r$-loop based at $x_{*}$. It is straightforward to verify that the map $\omega \mapsto (\omega(\frac i {n}))_{i=0}^{n}$ induces a homomorphism from the fundamental group $\pi_1(X, x_{*})$ to $\pi_{1,r}(X,x_{*})$; cf.~\cite{BCW}.

\begin{lem}\label{everythingComesFromTopology}
Let $G \acts Y$ be an action of a finitely generated group $G$ by homeomorphisms on a geodesic compact metric space $(Y,d)$, and let $r \geqslant 1$. If the action has a global fixed point~$y$, then the homomorphism $p\colon \pi_1(X, x_{*})\to \pi_{1,r}(X,x_{*})$ is surjective, where $X=(\{1\}\times Y,d_G)$ (corresponding to $t=1$ in \cref{defn:warpedcone}).
\end{lem}

\Cref{everythingComesFromTopology} can be deduced from \cite{VigFund}*{Theorem 42}, but we provide a direct proof which does not require the entire framework of \cite{VigFund}.

\begin{proof}[Proof of \cref{everythingComesFromTopology}]
Consider an $r$-loop $\omega = (x_0, \ldots, x_n)$ based at $x_{*}$. By \cite{roe}*{Proposition~1.6}, for every $i=0, \ldots, n-1$, there is a sequence $(x_i^j)_{j=0}^{n_i}$ with $x_i^0=x_i$ and $x_i^{n_i} = x_{i+1}$ such that
\begin{equation}\label{infimumAttainedRoe}
d_G(x_i, x_{i+1}) = \sum_{j=0}^{n_i-1} d_G(x_i^j, x_i^{j+1}),
\end{equation}
and furthermore, for every $j$, the distance $d_G(x_i^j, x_i^{j+1})$ equals $d(x_i^j, x_i^{j+1})$ or there is $g_i^j\in G$ such that  $d_G(x_i^j, x_i^{j+1})=|g_i^j|$ and $x_i^{j+1} = g_i^j x_i^{j}$.

By induction, the $r$-loop $\omega$ is $r$-homotopic to the sequence consisting of the same points $x_i$, but where each $x_i$ is repeated $n_i$ times. The supremum distance between the latter sequence and the sequence
\[
	\omega'=(x_0^0,\ldots, x_0^{n_0-1}, x_1^0,\ldots, x_1^{n_1-1}, \ldots, x_{n-1}^0,\ldots, x_{n-1}^{n_{n-1}-1}, x_n)
\]
is bounded by $r$ by \eqref{infimumAttainedRoe}, and hence $\omega$ and $\omega'$ are $r$-homotopic. We re-index $\omega'$ as $(y_0, \ldots, y_m)$. By construction, for $k=0,\ldots,m-1$, we have that $d_G(y_k, y_{k+1})$ equals $d(y_k, y_{k+1})$ or there is $g_k\in G$ such that $d_G(y_k, y_{k+1})=|g_k|$ and $y_{k+1} = g_k y_k$. If~$d_G(y_k,y_{k+1})=d(y_k,y_{k+1})$ for all $k=0,\ldots,m-1$, the statement of the lemma follows, since then $\omega'$ is an $r$-loop with respect to the metric $d$.

Otherwise, fix $k$ such that $d_G(y_k, y_{k+1})=|g_k|$ and $y_{k+1} = g_k y_k$. Since $(Y,d)$ is geodesic, there is an $r$-path $(y_k^0,\ldots, y_k^{m_k})$ (with respect to the metric $d$) with $y_k^0 = y_k$ and  $y_k^{m_k} = y$ (where $y$ is the global fixed point), and we can further require that $(g_k y_k^0,\ldots, g_k y_k^{m_k})$ is also an $r$-path with respect to $d$. Note that the sequence consisting of $y_k$ repeated $m_k$ times followed by $y_{k+1}$ repeated $m_k$ times is within the supremum distance $r$ from the sequence consisting of $y_k = y_k^0$, followed by $y_k^1$ repeated $m_k-1$ times, followed by $g_k y_k^1$ repeated $m_k-1$ times, and followed by $y_{k+1} =  g_k y_k^0$. Inductively, we can reach the $r$-path $(y_k^0, \ldots, y_k^{m_k}, g_k y_k^{m_k}, \ldots g_k y_k^{0})$. This path has the property that every two consecutive entries are within distance~$r$ with respect to~$d$; this uses the fact that $y_k^{m_k} = g_k y_k^{m_k}$.

By considering all $k$ as above, we see that $\omega'$ is $r$-homotopic to a sequence $\omega'' = (z_l)_{l=0}^N$ that is an $r$-loop with respect to the metric $d$. Since $Y$ is geodesic, there is a $1$-Lipschitz map $\alpha\colon [0,Nr]\to Y$
such that $\alpha(lr) = z_l$ for $l=0,\ldots,N$. Hence, the loop $\beta\colon \mathbb S^1\to X$ given by $\beta(\theta)= \alpha(\theta Nr)$ has the property that $p([\beta]) = [\omega''] = [\omega]$, which finishes the proof.
\end{proof}

\begin{rem}\label{noGlobalFixedPoint}
It is clear from the above proof that instead of a global fixed point, it suffices to assume that every element of $G$ of length at most $r$ has a fixed point, and by replacing it with a sequence of generators, one sees that it suffices that every generator has a fixed point. We formulated \cref{everythingComesFromTopology} in the more restrictive setting because all our actions have a global fixed point.
\end{rem}

Let $v \in \rho^{-1}\big((0,0)\big)$ for $\rho \colon Z_g \to \T^2$. This point is actually unique.

\begin{thm} \label{thm:ultimatestaircasetheorem}
Let $Z_g$ be an arbitrary staircase; see \cref{staircaseDef}. For every $t \in \N$ and $r \geqslant 1$, the discrete fundamental group $\pi_{1,r}((\{t\} \times Z_g,d_{\F_2}),v)$ is trivial. In particular:
\begin{enumerate}[(i)]
 \item \label{in_particular} an $\cO$-graph associated with the warped cone $\cO_{\F_2}Z_g$ is  coarsely non-equivalent  to any box space of a finitely generated group;
 \item \label{furthermore} such an $\cO$-graph is not coarsely equivalent to any $\cO$-graph associated with any warped cone $\cO_G M$ over a free action of an infinite group $G$ on a manifold $M$.
\end{enumerate}
\end{thm}

This result immediately implies \cref{fundamentalVanishing}.

\Cref{thm:ultimatestaircasetheorem} relies on two rather opposite properties of the action $\F_2 \acts Z_g$: the existence of fixed points and non-triviality at the homotopy level. More precisely, in \cref{everythingComesFromTopology} we used the fact that the action has a sufficient amount of fixed points (cf. \cref{noGlobalFixedPoint}) to show that the discrete fundamental group is an image of the fundamental group $\pi_1(Z_g,v)$. On the other hand, in the proof of \cref{thm:ultimatestaircasetheorem} we use the fact that the induced action of $\F_2$ on $\pi_1(Z_g,v)$ is sufficiently ``mixing'' for this image to be finite (actually trivial).

\begin{proof}[Proof of \cref{thm:ultimatestaircasetheorem}]
\begin{figure}
 \includegraphics[scale=0.5]{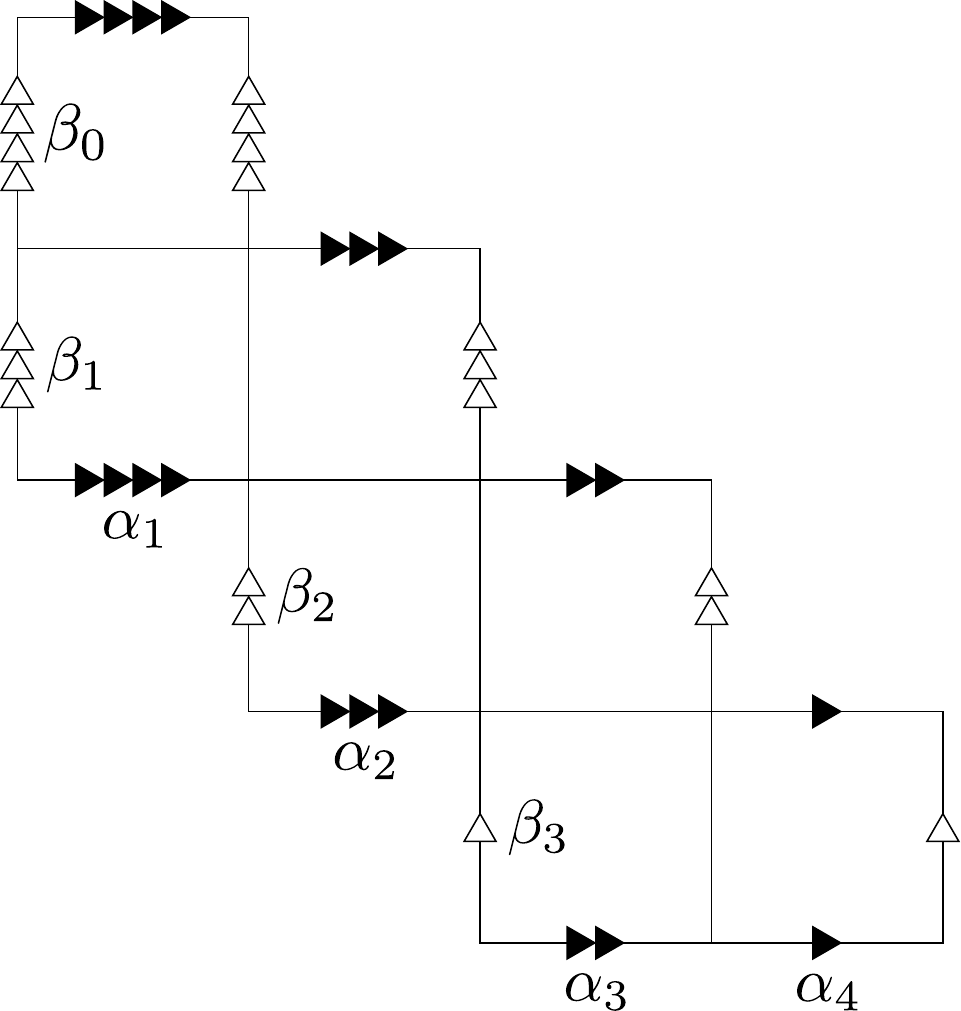}
 \caption{Staircase of genus 4 with labelled edges}
 \label{staircaseLabelled}
\end{figure}

It is straightforward to see that the staircase $Z_g$ can be viewed as a CW-complex with a single $0$-cell $v$, $1$-cells $\alpha_1, \ldots, \alpha_g, \beta_0, \ldots, \beta_{g-1}$, and a number of $2$-cells; see \cref{staircaseLabelled}. We can identify the $1$-cells with generators of $\pi_1(Z_g,v)$.

For notational convenience, let $a = a_2$ and $b=b_2$. Recall that $\F_2 = \langle a,b \rangle$.

First, note that $b_*(\alpha_1) = \beta_1 * \beta_0 * \alpha_1$; see \cref{staircaseLabelled}. If we fix a loop $\omega$ representing $\alpha_1$, then the supremum distance between $\omega$ and $b\circ \omega$ is at most 1 (as the metric is warped), and hence,
\[p(\alpha_1) = p([\omega]) = p([b\circ \omega]) = p(b_*(\alpha_1)) = p(\beta_1 * \beta_0 * \alpha_1).\]
Therefore,
\begin{equation}\label{C1}
p(\beta_1 * \beta_0) = 1;\;\text{that is,}\; p(\beta_1) = p(\beta_0)^{-1}.
\end{equation}

Define $\zeta_1=\beta_1 * \beta_0$, i.e., \eqref{C1} shows that $p(\zeta_1)=1$. For $i=2, \ldots, g-1$, we define
\[
	\zeta_i \defeq \beta_i * \alpha_{i-1}^{-1} * \beta_{i-1} *  \alpha_{i-2}^{-1} * \cdots * \beta_2 * \alpha_1^{-1} * \beta_1 * \beta_0 * \alpha_1 * \beta_0^{-1} *
\alpha_2 * \beta_1^{-1} * \cdots * \alpha_{i-1} * \beta_{i-2}^{-1}.
\]
Observe that
\begin{equation}\label{relationCi}
\zeta_i = \beta_i * \alpha_{i-1}^{-1} * \zeta_{i-1} * \alpha_{i-1} * \beta_{i-2}^{-1}.
\end{equation}
For $i=2, \ldots, g-1$, we have $b_*(\alpha_i) = \zeta_i * \alpha_i$ (see \cref{staircaseLabelled}), so $p(\alpha_i) = p(\zeta_i * \alpha_i)$, and hence, similarly as in \eqref{C1}, we get $p(\zeta_i) = 1$. Thus, by the triviality of $p(\zeta_{i-1})$ and by \eqref{relationCi}, we obtain
\[
	p(\beta_i) = p(\beta_{i-2}).
\]
Hence, we have proved that $p(\beta_i)$ equals either $p(\beta_1)$ or its inverse (depending on the parity) for each  $i=0,\ldots,g-1$.

By symmetry (replacing the action of $b$ with the action of $a$ and $\beta_i$ with $\alpha_{g-i}$), we get that $p(\alpha_i)$ equals either $p(\alpha_{g-1})$ or its inverse, depending on the parity.

By a similar argument, the action of $\F_2$ on $\pi_1(Z_g,v)$ can be used to show that the generators $p(\alpha_{g-1})$ and $p(\beta_{1})$ commute and are of order at most $2$, so
$\pi_{1,r}(\{t\}\times Z_g,v)$ is a quotient of $\mathbb Z_2 \oplus \mathbb Z_2$. We omit the details as the following argument (using the action on the fundamental groupoid instead) proves already that the generators are trivial.
Consider the path $\beta_{1/2}\colon[0,\frac{1}{2}] \to Z_g$ given by ``the first half'' of the edge $\beta_0$. Obviously, $\beta_{1/2} * \beta_{1/2}^{-1}\colon [0,1] \to Z_g$ is homotopically trivial. Since the midpoint of the edge $\beta_0$ is preserved by $a$, the function $\beta_{1/2} * (a\circ \beta_{1/2}^{-1})$ is still a continuous loop at $v$. Observe that $\beta_{1/2} * (a\circ \beta_{1/2}^{-1})$ belongs to the homotopy class $\alpha_1^{-1}$. As before, we note that the supremum distance between $\beta_{1/2} * \beta_{1/2}^{-1}$ and $\beta_{1/2} * (a\circ \beta_{1/2}^{-1})$ is at most~1, and hence,
\[1 = p(\beta_{1/2} * \beta_{1/2}^{-1}) =  p(\beta_{1/2} * (a\circ \beta_{1/2}^{-1})) = p(\alpha_1^{-1}),\]
so $p(\alpha_1^{-1})$, and hence $p(\alpha_{g-1})$, is trivial. By symmetry, $p(\beta_1)$ is trivial as well, and hence the whole group $\pi_{1,r}((\{t\}\times Z_g,d_{\F_2}),v)$ generated by $p(\alpha_{g-1})$ and $p(\beta_1)$ is trivial.

\smallskip

We now prove \eqref{in_particular}. It is known that whenever $r \geqslant 1$, $H$ is an infinite finitely generated group, and $(\cay(H/H_n, T_n))_n$ defines a box space of $H$, then $\pi_{1,r}(\cay(H/H_n, T_n))$ is infinite for sufficiently large~$n$. Indeed, \cite{DK}*{Lemma 3.4} gives a description of the discrete fundamental groups in the case of finitely presented groups, which was generalised to finitely generated groups  in \cite{VigFund}*{Theorem~58}.
Note that these results are formulated for a single Cayley graph $\cay(H/H_n, T_n)$ under the assumption that for a certain $R=R(r)$ the intersection $B_H(1_H, R)\cap H_n$ is trivial, so they still apply with our general definition of a box space, \cref{def:boxspace}.
Consequently, by \cref{surjectivityOfQI}, it is not possible for $(\cay(H/H_n, T_n))_n$ and an $\cO$-graph associated with $\cO_{\mathbb{F}_2}(Z_g)$ to be quasi-isometric with uniform constants. The same holds if one replaces ``quasi-isometric with uniform constants'' by ``coarsely equivalent with uniform control functions'', because both $\cO_{\mathbb{F}_2}(Z_g)$ and $(\cay(H/H_n, T_n))_n$ are quasi-geodesic.

\smallskip

We now prove \eqref{furthermore}. It again follows from \cref{surjectivityOfQI} because we claim that the discrete fundamental groups $\pi_{1,r}((\{t\} \times M, d_G),m_{*})$ are infinite, where $m_\ast$ is a basepoint in $\{t\} \times M$. Indeed,
if the acting group $G$ is finitely presented and the action $G\acts M$ is isometric, free, and minimal,
then $\pi_{1,r}((\{t\} \times M, d_G),m_{*})$ is a semi-direct product $\pi_1(M)\rtimes G$ by \cite{FNvL}*{Section~3}. In particular, it is infinite. More generally, if the action is free, then \cite{VigFund}*{Theorem~7.1} asserts that $\pi_{1,r}((\{t\} \times M, d_G),m_{*})$ is the quotient of a semi-direct product $\pi_1(M)\rtimes \F_S$ (where $S$ is the generating set of $G$)
by the normal closure of a certain set $K_r$ satisfying
\[
	K_r \subseteq \left\{ (\omega, g) \in \pi_1(M)\rtimes \F_S \;\big\vert\; |g|\leqslant 4r \text{ and }  q(g)=1_G \right\},
\]
where $q\colon F_S \to G$ is the quotient map.
In particular, $K_r \subseteq \pi_1(M) \rtimes \ker q$, and hence, $\pi_{1,r}((\{t\}\times M, d_G),m_{*}) \cong \pi_1(M)\rtimes \F_S / \langle\!\langle{K_r}\rangle\!\rangle$ surjects onto $G \cong \pi_1(M)\rtimes \F_S / (\pi_1(M) \rtimes \ker q)$, where $\langle\!\langle{K_r}\rangle\!\rangle$ denotes the normal closure of $K_r$.
\end{proof}

\begin{rem}\label{newTrivialFundamentalGroups}
The main application of the theory developed in \cite{VigFund} is the proof that, in our terminology, the  $\cO$-graphs associated to certain actions are not coarsely equivalent to any box space. The first main class of such actions (see \cite{VigFund}*{Corollaries~5 and~43}) are actions by rotations on odd-dimensional spheres. The second main class (see \cite{VigFund}*{Theorem~10 and Corollary~67}) consists of free isometric actions $G\acts M$ on a Riemannian manifold with finite fundamental group, under certain assumptions on $G$; an explicit example going back to \cite{Margulis} is $\SO(d, \Z[\frac 1 5])\acts \SO(d, \R)$ with $d\geqslant 5$.

In both cases, the space acted upon has finite fundamental group and the action is isometric, and in the latter case, it is also free. One of the novelties of \Cref{thm:ultimatestaircasetheorem} is that we show an analogous statement for a class of actions such that neither the space acted upon has finite fundamental group nor the action is isometric or free. Interestingly, \cref{thm:ultimatestaircasetheorem} shows not only coarse non-equivalence with box spaces, but also with warped cones over actions from the latter class from \cite{VigFund}.
\end{rem}

\section{Coarse non-equivalence of origami expanders as bundles over the original Margulis expander} \label{bundles}

\begin{conv}\label{conv:ASeqAsAMap} Since we often deal with sequences of spaces, we use the brief category-theoretic notation $f\colon X \to Y$ to denote a sequence ${(f_n\colon X_n\to Y_n)}_n$ of maps between sequences of spaces $X={(X_n)}_n$ and $Y={(Y_n)}_n$.
\end{conv}

In this section, we study the problem of distinguishing our origami expanders from each other. Instead of distinguishing them per se, we show that the sequence of maps $P^\rho\colon\cO_{\F_2} \Sigma \to \cO_{\SL_2(\Z)}\T^2$ induced by the branched covering $\rho\colon \Sigma\to \T^2$ is never a coarse equivalence. Furthermore, given another branched covering $\rho'\colon \Sigma'\to \T^2$ of an origami surface, we define the corresponding sequences $P^\rho$ and $P^{\rho'}$ to be \emph{coarsely equivalent} if there are coarse equivalences $l\colon \cO_{\SL_2(\Z)}\T^2 \to \cO_{\SL_2(\Z)}\T^2$ and $u\colon \cO_{\F_2} \Sigma \to \cO_{\F_2} \Sigma'$ such that the appropriate square commutes, namely the maps $l_t\circ P^\rho_t$ and $P^{\rho'}_t\circ u_t$ are close, uniformly in $t\in \N$. We show that there are infinitely many maps $P^\rho$ that are pairwise not coarsely equivalent.

\begin{thm}\label{notCEq}
For any origami $\Sigma$ equipped with the action from \cref{thm:F2}, the sequence $P^\rho\colon\cO_{\F_2} \Sigma \to \cO_{\SL_2(\Z)}\T^2$ is not a coarse equivalence.
\end{thm}

\begin{proof}
We will identify the free group $\F_2$ acting on $\Sigma$, with the subgroup of $\SL_2(\Z)$ generated by the matrices $a_k$ and $b_k$.
Since $\rho$ is equivariant with respect to the action $\F_2\acts \Sigma$ and the restriction of the action $\SL_2(\Z)\acts \T^2$ to $\F_2$, we get
\begin{equation}\label{inclusion}
\rho(\F_2 x) = \F_2\rho(x)
\end{equation}
for every $x \in \Sigma$.

Let $\gamma \in \SL_2(\Z) \smallsetminus \F_2$, and let $y\in \T^2$ be generic. Since by \cref{freeness on torus} the action of $\SL_2(\Z)$ on the orbit of $y$ is free, the points $y$ and $\gamma y$ lie in different $\F_2$-orbits. It follows from \eqref{inclusion} that for any $x\in \rho^{-1}(y)$ and $x'\in \rho^{-1}(\gamma y)$, the points $x$ and $x'$ lie in different orbits. By Remark~3.1 in \cite{Scompletions}, we have
\[\lim_{t\to \infty }d_{\F_2}((t,x), (t,x')) = \infty.\]
However, we have
\[d_{\SL_2(\Z)}(P^\rho_t(t,x), P^\rho_t(t,x')) = d_{\SL_2(\Z)}((t,y), \gamma (t, y)) \leqslant |\gamma|_{\SL_2(\Z)},\]
so the distance of the images under $P^\rho_t$ of the points $(t,x)$ and $(t,x')$ remains bounded as $t$ goes to infinity, showing that $P^\rho = \left(P^\rho_t\right)_{t\in \N}$ is not a coarse equivalence.
\end{proof}

For completeness, let us also note the following.

\begin{prop}\label{trivialButNecessary}
For any origami $\Sigma$ equipped with the action from \cref{thm:F2}, the sequence $P^\rho\colon\cO_{\F_2} \Sigma \to \cO_{\SL_2(\Z)}\T^2$ is bornologous.
\end{prop}
\begin{proof}
Up to coarse equivalence (in fact, up to bi-Lipschitz equivalence), the warped cone over an action does not depend on the choice of the generating set of the acting group \cite{roe}*{Proposition 1.7}. Therefore, we assume that the matrices $a_k$ and $b_k$ belong to the generating set of $\SL_2(\Z)$.

By~\cite{Scompletions}*{Proposition~2.1}, for every $R<\infty$ there exists $R'<\infty$ such that for every $t\in \N$ and every $x,x'\in \{t\}\times \Sigma \in \cO_{\F_2} \Sigma$ with $d_{\F_2}(x,x')\leqslant R$ there exists $\gamma\in \F_2$ such that $|\gamma| + d(\gamma x, x')\leqslant R'$. We claim $P^\rho$ is bornologous with a control function $\rho$ given by $R\mapsto R'$.

Indeed, let $t\in \N$ and $x,x'\in \{t\}\times \Sigma$, and denote $R\defeq d_{\F_2}(x,x')$. Let $\gamma \in \F_2$ be such that $|\gamma| + d(\gamma x, x')\leqslant R'$. We obtain
\begin{align*}
d_{\SL_2(\Z)}(P^\rho_t(x), P^\rho_t(x'))
&\leqslant |\gamma|_{\SL_2(\Z)} + d(\gamma P^\rho_t(x), P^\rho_t(x')) \\
&\leqslant |\gamma|_{\F_2} + d(P^\rho(\gamma x), P^\rho(x')) \\
&\leqslant |\gamma|_{\F_2} + d(\gamma x, x') \leqslant R'.\qedhere
\end{align*}
\end{proof}

\begin{defn}[Coarsely $n$-to-1 maps~\cite{MV}]\label{defn:nTo1} A bornologous sequence of maps $f_t\colon X_t\to Y_t$ is called \emph{coarsely $n$-to-1} for $n\in \N$ if for every $R<\infty$ there exists $S<\infty$ such that for every $t\in \N$ and $B\subseteq Y_t$ of diameter at most $R$ its inverse image can be decomposed as $f_t^{-1}(B) = \bigcup_{i=1}^n A_i$ with $A_i \subseteq X$ of diameter at most $S$.
\end{defn}

It is not hard to see that we may replace the sequence $(X_t)_t$ by a coarsely equivalent one, and the composition of the equivalences with the functions $f_t$ will still be coarsely $n$-to-$1$. Similarly, we may replace  the sequence $(Y_t)_t$ by a coarsely equivalent one.

We record the following easy observation.

\begin{lem}\label{productOfDegrees}\cite{MV}*{Proposition~3.3}
If $f$ and $g$ are coarsely $n$-to-1 and coarsely $m$-to-1 respectively, then the composition $g\circ f$ (if it makes sense) is coarsely $nm$-to-1.
\end{lem}

In \cite{MV}, \cref{defn:nTo1} and \cref{productOfDegrees} are formulated for a single map $f\colon X\to Y$, but \cref{productOfDegrees} holds mutatis mutandis in the case of sequences as in our \cref{defn:nTo1}.

Let us prove opposite estimates.

\begin{lem}\label{opposite-MV}
If $f\colon X\to Y$ is not coarsely $n$-to-1 for some $n\in \N$, then for any bornologous $g\colon Y \to Z$ the composition $g\circ f$ is not coarsely $n$-to-1.
\end{lem}
\begin{proof}
By the assumption, there exists $R<\infty$ such that for every $S<\infty$ there exists $t_S\in \N$ and a set $B_{S} \subseteq Y_{t_S}$ of diameter at most $R$ such that $f_{t_S}^{-1}(B_{S})$ cannot be covered by $n$ subsets of $X_{t_S}$ of diameter at most $S$.

Since $g$ is bornologous, there exists $R'<\infty$ such that for every $t\in \N$ the image under $g_t$ of any set $B \subseteq Y_t$ of diameter at most $R$ has diameter at most $R'$. Then, for $R'$ there is no $S'$ such that for every $t\in \N$ the inverse image via $g_t\circ f_t$ of every subset $C \subseteq Z_t$ of diameter at most $R'$ can be covered by $n$ subsets of $X_t$ of diameter at most $S'$. Indeed, for every $S'<\infty$ if we take $C=g_{t_{S'}}(B_{S'})$, then the inverse image $(g_{t_{S'}}\circ f_{t_{S'}})^{-1}(C)$ contains $f_{t_{S'}}^{-1}(B_{S'})$ and hence cannot be covered by $n$ subsets of $X$ of diameter at most $S'$.
\end{proof}

\begin{lem}\label{secondOpposite} Let $e\colon W\to X$ be bornologous and surjective, and let $f\colon X \to Y$ be bornologous. If the composition $f\circ e$ is coarsely $n$-to-1 for some $n\in \N$, then so is $f$.
\end{lem}
{The surjectivity assumption in \cref{secondOpposite} can be relaxed to coarse surjectivity, but we will not need it.}
\begin{proof}[Proof of \cref{secondOpposite}]
 Since $e$ is bornologous, there is a non-decreasing function $\rho$ such that $\diam e_t(U) \leqslant \rho(\diam U)$ for every $t\in \N$ and $U\subseteq W_t$ of finite diameter.
By the assumption, for every $R<\infty$ there is $S<\infty $ such that for every $B \subseteq Y_t$ of diameter at most $R$ the inverse image $(f_t\circ e_t)^{-1}(B)$ can be covered by $n$ subsets $A_1,\ldots, A_n$ of $W_t$ of diameter at most $S$. Then, $f_t^{-1}(B) = e_t\big((f_t\circ e_t)^{-1}(B)\big)$ can be covered by the sets $e_t(A_1),\ldots, e_t(A_n)$, each of which has diameter at most $\rho(S)$.
\end{proof}

Let us provisionally call the smallest integer $n$ such that a given sequence $f$ of maps is coarsely $n$-to-1 the \emph{coarse degree} of $f$ (we say that the coarse degree is infinite if no such $n$ exists).
\Cref{opposite-MV} shows that the coarse degree is non-decreasing under post-composition with bornologous maps. Similarly, \cref{secondOpposite} shows that the coarse degree is non-decreasing under pre-composition with bornologous surjective maps.

\begin{lem}\label{coarse-degree-generalised} Let $H\leqslant G$ be finitely generated groups and $G\acts Y$ be an action by Lipschitz homeomorphisms on a compact metric space $Y$ containing a free orbit. Then, the coarse degree of the sequence $P^{\id_{Y}}\colon \cO_{H}Y \to \cO_{G}Y$ equals the index $[G:H]$, where every function forming the sequence $P^{\id_{Y}}$ is the identity map of the underlying set.
\end{lem}
\begin{proof}
If we assume, as we may, that the generating set of $H$ is contained in the generating set of $G$, then clearly the warped metrics satisfy $d_G\leqslant d_H$, hence $P^{\id_{Y}}$ is bornologous.

Let us denote $n\defeq [G:H]\in \N \cup \{\infty \}$. Since subgroup inclusion is always a coarse embedding, there exists a non-decreasing function $c$ such that $|\gamma |_H \leqslant c(|\gamma |_G)$ for all $\gamma\in H$. Pick representatives $(\gamma_i)_{i=1}^n \in G^n$ of the right cosets $H\backslash G$.

First, we will show that the coarse degree of $P^{\id_Y}$ is at most $n$. For $n=\infty$ there is nothing to prove, so assume $n<\infty$ and let $L$ be the maximal length of~$\gamma_i$.
Let $R<\infty$. By Proposition~2.1 in \cite{Scompletions}, there exists $R'<\infty$ such that for every $x,x'\in \{t\}\times Y \in \cO_G Y$ with $d_{G}(x,x')\leqslant R$ there exists $\gamma\in G$ such that $|\gamma|_G + d(\gamma x, x')\leqslant R'$. Now, let $1\leqslant i\leqslant n$ be such that $\gamma_i$ belongs to the same coset as $\gamma$. Then $\gamma\gamma_i^{-1}$ has length at most $R'+L$ with respect to the generating set of $G$, and hence at most $R''\defeq c(R'+L)$ with respect to the generating set of $H$. We put $S_0=R'+R''$.

Let $t\in \N$ and $B \subseteq \{t\} \times Y \in \cO_G Y$ have diameter at most $R$, and pick $x \in B$. By the above, for every $x' \in B$, there exists $1\leqslant i\leqslant n$ and $\gamma\in G$ such that $\gamma \gamma_i^{-1}  \in H$,  $|\gamma \gamma_i^{-1}|_H \leqslant R''$, and $d(\gamma \gamma_i^{-1} x_i, x') = d(\gamma x, x') \leqslant R'$ for $x_i=\gamma_i x$. It follows that $d_H(x_i, x') \leqslant |\gamma \gamma_i^{-1}|_H + d(\gamma \gamma_i^{-1} x_i, x') \leqslant S_0$. It follows that the set $B$ is contained in the union of the balls around $x_i$ of radius $S_0$ with respect to the metric $d_H$, so taking $S=2S_0$ finishes the proof that $P^{\id_Y}$ is a coarsely $n$-to-$1$.

It remains to show that the coarse degree of $P^{\id_Y}$ is at least $n$, that is, that $P^{\id_Y}$ is not coarsely $m$-to-1 for any $m<n$. This is similar to the proof of \cref{notCEq}. Let~$L'$ be the maximal length of $\gamma_i$ for $i\leqslant m+1$, and let $y\in Y$ be a point whose orbit is free. Note that for any $t\in \N$ the set $B_t = \{ (t, \gamma_i y) : 1\leqslant i\leqslant m+1\} \subseteq \{t\}\times Y \in \cO_G Y$ is contained in the ball of radius $L'$ around $(t,y)$ so it has diameter at most $2L'$. On the other hand, since the action of $G$ on the orbit of $y$ is free and $\gamma_i$ belong to distinct cosets of $H$, it follows that $d_H((t, \gamma_i y), (t, \gamma_j y))$ goes to infinity with~$t$ for $i\neq j$. Hence, for every $S<\infty$ there exists $t<\infty$ such that any non-singleton subset of $B_t$ has diameter larger than $S$, and hence the only way to cover $(P_t^{\id_Y})^{-1}(B_t) = B_t \subseteq \{t\}\times Y \in \cO_H Y$ with (non-empty) sets of diameter at most $S$ is to cover it by $m+1$ singletons.
\end{proof}

\begin{cor}\label{degreeOfTheIdentity}
Let $k\geqslant 2$ and let $\Sigma$ be the torus equipped with the action of $\F_2$ given by the matrices $a_k$ and $b_k$. The coarse degree of $P^{\id_{\T^2}}\colon \cO_{\F_2}\Sigma\to \cO_{\SL_2(\Z)}\T^2$ equals $[\SL_2(\Z) : \langle a_k, b_k \rangle]$.
\end{cor}
\begin{proof}
By \cref{freeness on torus}, the action $\SL_2(\Z)\acts \T^2$ has many free orbits, and hence the claim follows immediately from \cref{coarse-degree-generalised}.
\end{proof}

Denote the sequence $\cO_{\F_2} \Sigma \to \cO_{\F_2}\T^2$ of maps induced by the branched covering $\rho\colon \Sigma\to \T^2$ by $Q^\rho$ (so that $P^\rho = P^{\id_{\T^2}} \circ Q^\rho$). For a permutation $\sigma$ of a finite set $X$, let $c(\sigma)$ denote the number of cycles in the cycle decomposition of $\sigma$ (in other words, it is the number of orbits of the action of $\langle \sigma \rangle$ on $X$).

\begin{lem}\label{degreeMain}
Let $\Sigma$ be an origami with origami datum $(m,\sigma,\tau)$ equipped with the action from \cref{thm:F2}, and let $\rho\colon \Sigma\to \T^2$ denote the corresponding branched covering map. Then the coarse index $\iota$ of $Q^\rho$ satisfies
\[
	\max(c(\sigma), c(\tau)) \leqslant \iota \leqslant m.
\]
\end{lem}
\begin{proof}
The proof that $Q^\rho$ is bornologous is identical to the proof of \cref{trivialButNecessary}, and thus we omit it.

We will first show that $\iota \leqslant m$. Let $R<\infty$. As before, there exists $R'<\infty $ such that for every $t\in \N$ and every $y,y'\in \{t\}\times \T^2 \in \cO_{\F_2} \T^2$ with $d_{\F_2}(y,y')\leqslant R$ there exists $\gamma\in \F_2$ such that $|\gamma| + d(\gamma y, y')\leqslant R'$. Then, since the maps $Q^\rho_t$ are clearly equivariant, we have $\gamma (Q_t^\rho)^{-1}(y) = (Q_t^\rho)^{-1}(\gamma y)$. Now, for every $x' \in (Q_t^\rho)^{-1}(y')$ there exists $\tilde x \in (Q_t^\rho)^{-1}(\gamma y)$ such that $d(\tilde x, x') = d(\gamma y, y')$ ($\tilde x$ need not be unique). Denote $x\defeq \gamma^{-1}\tilde x \in \gamma^{-1}(Q_t^\rho)^{-1}(\gamma y) = (Q_t^\rho)^{-1}(y)$. Summarising, for every $x' \in (Q_t^\rho)^{-1}(y')$ there exists $x \in (Q_t^\rho)^{-1}(y)$ such that
\[d_{\F_2}(x,x') \leqslant |\gamma| + d(\gamma x, x') = |\gamma| + d(\gamma y, y')\leqslant R'.\]
It follows that for every $B\subseteq \{t\}\times \T^2 \in \cO_{\F_2} \T^2$ of diameter at most $R$, we can pick $y\in B$, and then $(Q_t^\rho)^{-1}(B)$ is contained in the union of the balls of radius $R'$ around the elements of the fibre $(Q_t^\rho)^{-1}(y)$, whose cardinality is clearly at most $m$.

It remains to show $\max(c(\sigma), c(\tau)) \leqslant \iota$, and by symmetry it suffices to show $c(\tau) \leqslant \iota$. Enumerate the orbits of $\tau$ by the numbers $1,\ldots, c(\tau)$, and for every $j\in\{1,\ldots, c(\tau)\}$ pick a representative $i\in \{1,\ldots,m\}$ of the $j$th orbit, and let $x_j\defeq (1/k, 0)\in P_i$ be a point of the bottom edge of the $i$th square $P_i$. Note that every point $x_j$ is fixed by $a_k \in \F_2$, and the element $b_k \in \F_2$ permutes the finite set $\{(1/k, 0)\in P_{\tau^n(i)} : n\in \Z \}$, which is in fact the orbit of $x_j$ (if we replace $1/k$ by some multiple $l/k$ for $l\in \N_+$, the corresponding set may be a union of many orbits, which would even improve our lower bound on $\iota$). It follows that for $j\neq j'$, both from $\{1,\ldots, c(\tau)\}$, the points $x_j$ and $x_{j'}$ belong to different orbits, and hence the distance $d_{\F_2}((t,x_j), (t,x_{j'}))$ goes to infinity with~$t$ despite the fact that they have the same image $(t, (1/k, 0))$ under $Q_t^\rho$. Now, as in the proof of \cref{coarse-degree-generalised}, it follows that the coarse degree of $Q_t^\rho$ is at least the cardinality of the set $\{x_j : j\in\{1,\ldots, c(\tau)\} \}$, namely $c(\tau)$.
\end{proof}

Under certain assumptions, \cite{Spiecewise}*{Theorem~4.1} proves that when a coarse equivalence $\cO_G X \to \cO_H Y$ is induced by a surjective map $\rho\colon X\to Y$, then $H$ is a quotient of $G$ by a finite kernel and $\rho$ is equivariant. Hence, maps of warped cones induced by equivariant maps of surfaces are promising candidates for coarse equivalences, even though \cite{Spiecewise}*{Theorem~4.1} does not directly apply to our actions on origami surfaces.
However, the following result shows that the converse to \cite{Spiecewise}*{Theorem~4.1} does not hold, as in this case the acting groups coincide and $\rho$~is equivariant.

The statement is analogous to \cref{notCEq}.

\begin{thm}
Let $\Sigma$ be an origami with origami datum $(m,\sigma,\tau)$ equipped with the action from \cref{thm:F2}, and let $\rho\colon \Sigma\to \T^2$ denote the corresponding branched covering map. Suppose that at least one of $\sigma$ and $\tau$ is not the full cycle on $\{1,\ldots,m\}$. Then $Q^\rho\colon \cO_{\F_2} \Sigma \to \cO_{\F_2}\T^2$ is not a coarse equivalence.
\end{thm}
\begin{proof}
The assumption that at least one of $\sigma$ and $\tau$ is not the full cycle is equivalent to the inequality $\max(c(\sigma), c(\tau)) > 1$, so \cref{degreeMain} shows that $Q^\rho$ is not coarsely $1$-to-$1$. However, every coarse embedding is coarsely $1$-to-$1$ (in fact, also the converse holds).
\end{proof}

\begin{prop}\label{finalDegreeEstimate}
Let $\Sigma$ be an origami with origami datum $(m,\sigma,\tau)$ equipped with the action from \cref{thm:F2}, and let $\rho\colon \Sigma\to \T^2$ denote the corresponding branched covering map. Then the coarse index $\iota$ of $P^\rho$ satisfies
\[ \max(c(\sigma), c(\tau), [\SL_2(\Z) : \langle a_k, b_k \rangle]) \leqslant \iota \leqslant m \cdot [\SL_2(\Z) : \langle a_k, b_k \rangle].\]
\end{prop}
\begin{proof}
Note that $P^\rho = P^{\id_{\T^2}}\circ Q^\rho$. Hence, the inequality $\max(c(\sigma), c(\tau)) \leqslant \iota$ follows from the first inequality of \cref{degreeMain} together with \cref{opposite-MV} (in the light of \cref{trivialButNecessary}). The inequality $[\SL_2(\Z) : \langle a_k, b_k \rangle] \leqslant \iota$ follows from \cref{degreeOfTheIdentity} by \cref{secondOpposite}.

The second inequality follows from the second inequality of \cref{degreeMain} and from \cref{degreeOfTheIdentity} by \cref{productOfDegrees}.
\end{proof}

\begin{rem}In particular, the coarse degree of $P^\rho$ is finite if and only if $k=2$ because otherwise
$[\SL_2(\Z) : \langle a_k, b_k \rangle] = \infty$.
\end{rem}

\begin{ex}\label{infinitelyManyDegreesForStaircases} For $\Sigma=Z_g$ (see \cref{staircaseDef}), \cref{finalDegreeEstimate} gives
\[
	\max(g, 12) \leqslant \iota \leqslant 12(2g-1).
\]
\end{ex}

Suppose that $f$ is bornologous and $e, e'$ are coarse equivalences such that the composition $e'\circ f\circ e$ makes sense. Then, as remarked above, for every $n\in \N$ the map $f$ is coarsely $n$-to-1 if and only if $e'\circ f\circ e$ is coarsely $n$-to-1. It is also easy to verify that if $f,f'\colon X\to Y$ are close, then $f$ is coarsely $n$-to-1 if and only if $f'$ is such. Hence, if two sequences $P^\rho$ and $P^{\rho'}$ of maps are coarsely equivalent as defined at the beginning of \cref{bundles}, then their coarse degrees must coincide.

\begin{thm}\label{InfinitelyManyBundles}
There are infinitely many pairwise  non-coarsely-equivalent sequences $P^\rho\colon\cO_{\F_2} \Sigma \to \cO_{\SL_2(\Z)}\T^2$ induced by the branched coverings $\rho\colon \Sigma\to \T^2$ for different origamis $\Sigma$.
\end{thm}
\begin{proof}
By the discussion above, the coarse degree of a sequence $P^\rho\colon\cO_{\F_2} \Sigma \to \cO_{\SL_2(\Z)}\T^2$ is preserved under coarse equivalence. The claim thus follows immediately from the fact that by \cref{infinitelyManyDegreesForStaircases} the set of the possible finite values of the coarse degree of sequences $P^\rho$ is infinite.
In fact, for $I=\{24^n : n\in \N\}$ and any distinct $g,g'\in I$, the coarse degrees of the corresponding sequences of maps $P^{\rho}\colon\cO_{\F_2} \Sigma_{g} \to \cO_{\SL_2(\Z)}\T^2$ and $P^{\rho'}\colon\cO_{\F_2} \Sigma_{g'} \to \cO_{\SL_2(\Z)}\T^2$ are distinct, so all such sequences are not coarsely equivalent.
\end{proof}

\Cref{notCEq,InfinitelyManyBundles}
imply \cref{pairwiseDistinctExpanderBundles}.

\begin{rem}
 The theorem immediately implies an analogous statement for $\cO$-graphs, namely: There are infinitely many pairwise  non-coarsely-equivalent sequences of maps from $\cO$-graphs to the original Margulis expander, where the $\cO$-graphs are associated with affine actions on origamis.
\end{rem}

\end{document}